\documentclass[11pt,reqno]{amsart}
\usepackage{amssymb}
\usepackage{amsmath}
\usepackage{geometry}                % See geometry.pdf to learn the layout options. There are lots.
\usepackage{graphicx}
\usepackage{epstopdf}
\usepackage{mathrsfs}
\usepackage{xcolor}
\usepackage{mathtools}
\usepackage{comment}
\mathtoolsset{showonlyrefs}
\newcommand{\Addresses}{{% additional braces for segregating \footnotesize
  \bigskip
  \footnotesize
 
 \textsc{Department of Mathematics, KTH Royal Institute of Technology,
  Stockholm, Sweden}\par\nopagebreak
  \textit{E-mail addresses:} \quad \texttt{persj@kth.se, jostromb@kth.se}
}}

\title{Schr\"odinger means in higher dimensions}
\author{Per Sj\"olin and Jan-Olov Str\"omberg}
\renewcommand{\S}[0]{{\bf S}}

\newcommand{\R}[0]{{\mathbb R}}
\newcommand{\C}[0]{{\mathbb C}}

\newcommand{\supp}[0]{\mbox{supp }}
\newcommand{\e}[0]{\mbox{e}}

\newtheorem{theorem}{Theorem}

\newtheorem{lemma}{Lemma}
\newtheorem{proposition}{Proposition}
\newtheorem{corollary}{Corollary}
\theoremstyle{definition}

\renewcommand{\smallskip}{}

\newcommand{\aname}{\em}
\newcommand{\jname}{}

\begin{document}
\begin{abstract}
Maximal estimates for Schr\"odinger means and convergence almost everywhere of sequences of Schr\"odinger means are
studied.
\end{abstract}
\let\thefootnote\relax 
\maketitle
\section{Introduction}
For  $f\in L^2(\R^n), n\ge1$ and $a>0$ we set 
\begin{align}
\label{eq:Fouriertransform}
\hat f(\xi)=\int_{\R^n} \e^{-i\xi\cdot x} f(x)\,dx, \xi \in \R^n,
\end{align}
and \[ 
S_tf(x)=(2\pi)^{-n} \int_{\R^n} \e^{i\xi\cdot x} \e^{it|\xi|^a}\hat f(\xi)\, d\xi,\quad x\in\R^n, t\ge0.
\]
We introduce Sobolev spaces $H_s=H_s(\R^n)$ by setting 
\[
H_s=\{f\in \mathscr{S}^\prime; \|f\|_{H_s} <\infty \}, s\in\R,
\]
where
\[
 \|f\|_{H_s}=\left(\int_{\R^n} (1+[\xi[^2)^s |\hat f(\xi)|^2\,d\xi \right)^{1/2}
\]
and $ \mathscr{S}= \mathscr{S}(\R^n)$ denotes the Schwartz class.
We let the sequence $\{t_m\}_1^\infty$ have the properties that
\begin{align}
\label{eq:1}
1>t_1>t_2>t_3>\dots>0 \mbox{ and }\lim_{m\to\infty}t_m=0.
\end{align}  
We shall study the problem of deciding for which sequences  $\{t_m\}_1^\infty$ and functions $f$ one has
\begin{align}
\label{eq:seqlimit}
\lim_{m\to\infty} S_{t_m} f(x)=f(x)\mbox{ almost everywhere}.
\end{align}
For $r>0$ we say that $\{t_m\}_1^\infty\in l^r$ if $\sum\limits_1\ t_m^r<\infty$, and that  $\{t_m\}_1^\infty \in l^{r,\infty}$ if
 $\#\{m;t_m>b\}$\\$\lesssim b^{-r}$ for $b>0$.\\
\par
In Sj\"olin and Str\"omberg \cite{Sjo-Str1} we proved that if $a>0, n\ge1, 0<s<a/2$, and 
 $\{t_m\}_1^\infty\in l^r$ for some $r<2s/(a-2s)$, then \eqref{eq:seqlimit} holds if $f\in H_s(\R^n)$.\\
 In the case $n=1$ Dimou and Seeger \cite{Dim-See} have proved that if $a>0, a\ne1, 0<s<a/4$ and  $\{t_m\}_1^\infty\in l^r$,  where
  $r=2s/(a-4s)$, then \eqref{eq:seqlimit} holds if $f\in H_s(\R^)$. They also proved that here $r$ cannot be replaced by $r_1$
  if $r_1>r$.\\
  In Section 3 of this paper we shall improve the result from \cite{Sjo-Str1} mentioned above in the case $n\ge2$. We shall prove that
   \eqref{eq:seqlimit}  holds for $f\in H_s(\R^n)$ if $a>0,a\ne1, n\ge2, 0<s<a/2$, and   $\{t_m\}_1^\infty \in l^{r,\infty}$
  where  $r=2s/(a-2s)$. In the proof of this result we shall use the following theorem.
  \begin{theorem}
  Assume $a>0,n\ge1,\lambda\ge1 $, and let the interval $J\subset[0,1]$. Assume also that $f\in L^2(\R^n)$ and 
  $\supp\hat f\subset B(0;\lambda)=\{\xi\in\R^n;|\xi|\le\lambda\}$. Then one has\[
  \|\sup\limits_{t\in J}|S_tf| \|_2\le(1+C|J|^{1/2}\lambda^{a/2})\|f\|_2.
  \]
  \end{theorem}
  Let the ball $B$ be a subset of $R^n$ and let the interval  $J\subset[0,1]$. Set $E=B\times J$. In Section 2 we shall also study maximal functions of the type\[
  S_E^*f(x)=\sup\limits_{(y,t)\in E}|S_tf(x+y)|,\, x\in\R^n, f\in\mathscr{S}(\R^n).
  \]
  In Section 4 we study relations between maximal estimates in one variabel and maximal estimates in dimension $n\ge2$.\\
  In Section 5 finally we giv a counter-example which shows that in the case $n\ge2, 0<s<a/4, r=2s/(a-4s), \{t_m\}_1^\infty \in l^{r,\infty},
  (t_n-t_{n+1})_1^\infty$ decreasing, there is no estimate \[
  \|\sup\limits_m|S_{t_m}f|\|_2\lesssim\|f\|_{H_s}
  \]
  for all radial functions $ f\in\mathscr{S}(\R^n)$.
  \par
  We write $A\lesssim B$ if there is a positive constant $C$ such that $A\le CB$, , and we write $A\sim B$ if  $A\lesssim B$ and $B\lesssim A$.
  \setcounter{theorem}{0}
  \section{Maximal estimates}
  The following theorem was proved in Sj\"olin and Str\"omberg \cite{Sjo-Str1}. We shall here give a new proof which is simpler than the proof in
  \cite{Sjo-Str1}.\\[.5cm]
\par
 {\bf Theorem A. }{\em (Sj\"olin-Str\"omberg) Asssume $a>0, n\ge1, \lambda\ge1$, and $J$ interval in $\R$. Assume  also that $f\in L^2(\R^n)$
 and $\supp\hat f\subset B(0;\lambda)$.\\
 Then one has\[
 \|\sup\limits_{t\in J}|S_tf| \|_2\le (1+C|J|\lambda^a)\|f\|_2.
 \]
 }%end theorem A
 \begin{proof}
 We can write\[
 S_tf(x)=c\int \e^{i\xi\cdot x} \e^{it|\xi|^a} \hat f(\xi)\, d\xi
 \]
 and $J=[t_0,t_0+r]$ where $r=|J|$.\\
 We have \[
  \e^{it|\xi|^a}=\Delta+\e^{it_0|\xi|^a}
 \]
 where $\Delta=\e^{it|\xi|^a}-\e^{it_0|\xi|^a}$. It follows that
 \[
 \Delta=i|\xi|^a\int_{t_0}^t\e^{i|\xi|^as}\,ds,
 \]
 and
 \begin{align}
 S_tf(x)&=c\int_{\R^n} \int_{t_0}^t \e^{i\xi\cdot x} i|\xi|^a\e^{is|\xi|^a} \hat f(\xi)\, d\xi\,ds+c\int_{\R^n} \e^{i\xi\cdot x} \e^{it_0|\xi|^a} \hat f(\xi)\, d\xi\\
 &=S_1(x,t)+S_2(x,t).
 \end{align}
 Hence \[
 |S_1(x,t)|\lesssim     \int_{t_0}^t \left|\int_{\R^n}\e^{i\xi\cdot x} i|\xi|^a\e^{is|\xi|^a} \hat f(\xi)\, d\xi\right|\,ds
 \]
 and \[
 \sup\limits_{t\in J}|S_1(x,t)|\lesssim     \int_{t_0}^{t_0+r} \left|\int_{\R^n}\e^{i\xi\cdot x} i|\xi|^a\e^{is|\xi|^a} \hat f(\xi)\, d\xi\right|\,ds.
 \]
 Using Minkovski's inequality and Plancherel's theorem we obtain
 \begin{align}
 \left( \int\sup\limits_{t\in J}|S_1(x,t)|^2\,dx\right)^{1/2}&\lesssim  \int_{t_0}^{t_0+r} \||\xi|^a\hat f(\xi)\|_2\,ds\\
&\le r\lambda^a\left(\int_{B(0;\lambda)}|\hat f(\xi)|^2|\right)^{1/2}\lesssim r\lambda^a \|f\|_2.
 \end{align}
 Also\[
  \left( \int\sup\limits_{t\in J}|S_2(x,t)|^2\,dx\right)^{1/2}=\left( \int|S_2(x,t_0)|^2\,dx\right)^{1/2}=\|f\|_2,
 \]
 and we obtain \[
 \|\sup\limits_{t\in J}|S_tf|\|_2\le (1+Cr\lambda^a)\|f\|_2,
 \]
 which proves the theorem.
 \end{proof}
\par
In the following theorems let $J$ denote an interval. In the one-dimensional case Dimou and Seeger obtained the following result.\\[.5cm]
{\bf Theorem B. }{\em (Dimou -Seeger  )
Assume $a>0, a\ne1,n=1,\lambda\ge1$. and $J\subset[0,1]$. Assume also that $f\in L^2(\R)$ and 
$\supp \hat f\subset\{\xi;\lambda/2\le|\xi|\le\lambda\}$.\\
Then one has\[
\|\sup\limits_{t\in J}|S_tf|\|_2\lesssim \left(1+|J|^{1/4}\lambda^{a/4}\right)\|f\|_2.
\]\\[.5cm]
} %end theorem B
We have the following result.
\par
 \begin{theorem}
  Assume $a>0,n\ge1,\lambda\ge1 $, and let the interval $J\subset[0,1]$. Assume also that $f\in L^2(\R^n)$ and 
  $\supp\hat f\subset B(0;\lambda)=\{\xi\in\R^n;|\xi|\le\lambda\}$. Then one has\[
  \|\sup\limits_{t\in J}|S_tf| \|_2\le(1+C|J|^{1/2}\lambda^{a/2})\|f\|_2.
  \]
  \end{theorem}
\begin{proof}
First we study the case $|J|\lambda^a\le1$. Then $(1+C|J|\lambda^a))\le(1+C|J||^{1/2}\lambda^{a/2})$. Applying Theorem A we get\[
\|\sup\limits_{t\in J}|S_tf|\|_2\le(1+C|J|\lambda^a))\|f\|_2\le(1+C|J||^{1/2}\lambda^{a/2})\|f\|_2.
\]
It remains to study the case $|J|\lambda^a>1$.\\
Cover $J$ with intervals $J_i, i=1,2,\dots,N$, of length $\lambda^{-a}$. We may take $N\le|J|\lambda^a+1$.. Using Theorem A we obtain
\begin{align}
 \|\sup\limits_{t\in J}|S_tf|\|_2^2\le\sum\limits_{i=1}^N\|\sup\limits_{t\in J_i}|S_t|f\|_2^2\lesssim\sum\limits_{i=1}^N\left(1+|J_i|\lambda^a \right)\|f\|_2^2
 \lesssim N\|f\|^2_2\le\left(|J|\lambda^a+1\right)\|f\|^2_2
 \end{align} 
 and the inequality in the theorem follows.\\
 \end{proof}
 We have the following extension of Theorem B.
 \begin{theorem} 
Assume $a>0, a\ne1, \lambda\ge1$ and $J\subset[0,1]$. Assume also that $f\in L^2(\R)$ with $\supp \hat f\subset B(0,\lambda)$. Then one has\[
\|\sup\limits_{t\in J}|S_tf|\|_2\lesssim C\left( |J|^{1/4}\lambda^{a/4}+1 \right)\|f\|_2.
\]
 \end{theorem}
\begin{proof}
First we study the case $|J|\lambda^a\le1$. From Theorem A we obtain\[
 \|\sup\limits_{t\in J}|S_tf| \|_2\lesssim\|f\|_2.
\]
which proves the theorem in this case.\\
We then consider the case  $|J|\lambda^a>1$,  i.e. $\lambda>|J|^{-1/a}$. \\
We choose $k$ such that $2^{-k-1}\lambda<|J|^{-1/a}\le 2^{-k}\lambda$ and then write $f=\sum\limits_{j=0}^kf_j+g$ where 
$\supp\hat f_j\subset\{\xi;2^{-j-1}\lambda\le|\xi|\le2^{-j}\lambda\}$ for $j=0,1,\dots,k$, and $\supp\hat g\subset B(0;2^{-k-1}\lambda)$.\\
From Theorem A we conclude that\[
\|\sup\limits_{t\in J}|S_tg|\|_2\lesssim\left(1+|J||J|^{-1}\right)\|g\|_2\lesssim\|f\|_2
\]
and it follows from Theorem B that\[
\|\sup\limits_{t\in J}|S_tf_j|\|_2\lesssim\left(1+|J|^{1/4}2^{-ja/4}\lambda^{a/4}\right)\|f\|_2\lesssim|J|^{1/4}2^{-ja/4}\lambda^{a/4}\|f\|_2,
\]
since {$|J|2^{-ja}\lambda^a \ge1$}. Hence we have\[
\sum\limits_{j=0}^k\|\sup\limits_{t\in J}|S_tf_j|\|_2\lesssim|J|^{1/4}\lambda^{a/4}\|f\|_2
\]
and the theorem follows from the above estimates.
\end{proof}
\par
The method of Dimou and Seeger to prove Theorem B can be extended to all dimensions $n$ an gives the following result.
\begin{lemma}
Assume $n\ge1, a>0$ and $\lambda\ge1$.
Let the interval $J\subset[0,1]$, let $f\in L^2(\R^n)$  with $\supp\hat f\subset\{\lambda/2\le|\xi|\le\lambda\}$.\\
Then one has \[
 \|\sup_{t\in J}|S_tf|\|_2\le C\left( |J|^{n/4}\lambda^{na/4}+1 \right)\|f\|_2\mbox{when }a\ne1
 \]
 and 
 \[
 \|\sup_{t\in J}|S_tf|\|_2\le C\left( |J|^{(n+1)/4}\lambda^{(n+1)/4}+1 \right)\|f\|_2\mbox{ when }a=1
 \]
\end{lemma}
\par
We shall then study a more general problem. Let $E$ denote a bounded set in $\R^{n+1}$. For $f\in\mathscr{S}(\R^n)$ we introduce
 the maximal function\[
 S_E^*f(x)=\sup_{(y,t)\in E}\,|S_t f(x+y)|\,,\quad x\in\R^n.
\]
The method used to prove Lemma 1 can also be used to prove the following result.
\begin{lemma}
Assume $a>0, n\ge1$ and $\lambda\ge1$.
Let the interval $J\subset[0,1]$. Let $B$ be a ball in $\R^n$ with radius $r$, let $E=B\times J=\{(x,t);x\in B,t\in J\}$ and let $f\in L^2(\R^n)$ 
with $\supp\hat f\subset\{\lambda/2\le|\xi|\le\lambda\}$.\\ 
Then one has\[
 \|S_E^*f\|_2\le C\left( |J|^{n/4}\lambda^{na/4}+r^{n/2}\lambda{n/2}+1 \right)\|f\|_2\,\mbox{ when }a\ne1,
 \]
 and 
 \[
 \|S_E^*f\|_2^2\le C\left( |J|^{(n+1)/4}\lambda^{(n+1)/4}+r^{n/2}\lambda^{n/2}+1 \right)\|f\|_2\,\mbox{ when }a=1.
 \]
\end{lemma}
We observe that Lemma 1 is a special case of Lemma 2 by taking $B=B(0,\epsilon)$ with $\epsilon>0$ small enough.\\
\begin{proof}[Proof of Lemma 2]
We let  $\chi$ denote a smooth non-negative function on $\R$
supported on $[1/3,4/3]$ and identically $1$ on $[1/2,1]$. We also use
the same notation for the radial function on $\chi$ $\R^n$ with $\chi(\xi)=\chi(|\xi|)$.\\
We set \[
S_\lambda f(x,t)=\int_{R^n}\e^{i(x\cdot\xi+t|\xi|^a)} \hat f(\xi)\chi(\xi/\lambda)\,d\xi\,\quad x\in\R^n, t\in\R.
\]
We introduce measurable functions $t:\R^n\to J$ and $b:\R^n\to B$ and then have
\begin{align}
S_\lambda f(x+b(x),t(x))&=\int_{R^n}\e^{i(x+b(x))\cdot\xi+t(x)|\xi|^a)} \hat f(\xi)\chi(\xi/\lambda)\,d\xi\\
&=[\eta=\xi/\lambda]=\int_{\R^n}
\e^{i(\lambda (x+b(x))\cdot\eta+t(x)\lambda^a|\eta|^a)}\hat f(\lambda\eta)\chi(\eta)\,d\eta\,\lambda^n\\
=\lambda^n T_\lambda(\hat f(\lambda \cdot))(x)
\end{align}
where
\begin{align}
T_\lambda g(x)=\int\e^{i(\lambda (x+b(x))\cdot\xi+t(x)\lambda^a|\xi|^a)}g(\xi)\chi(\xi)\,d\xi.
\end{align}
We have\(\|\hat f(\lambda\cdot)\|_2=c\lambda^{-n/2}\|f\|_2\)and
$T_\lambda T^*_\lambda$ has kernel
\begin{align}
K_\lambda(x,y)=\int_{\R^n}\e^{i[\lambda(x-y+b(x)-b(y))\cdot\xi+\lambda^a(t(x)-t(y))|\xi|^a]}\chi(\xi)^2\,d\xi
\end{align}
We will majorize the kernel $K_\lambda$ by a convolution kernel $G_\lambda$, that is  $|K(x,y)|\lesssim G(x-y)$.
One then has the $L^2$-operator norm  $\|T_\lambda T^*_\lambda\|\lesssim\|G\|_1$.\\
First we have the trivial estimate \[
|K_\lambda(x,y)|\lesssim1
\]
 holds for all $x$ and $y$. We shall use this estimate when $\lambda|x-y|\le C_0+2\lambda d$ where $d=2r$.\\
 %which will refer to as Case 0.\\
For  $\lambda|x-y|> C_0+2\lambda d$ we have \[
    |x-y| \left(1-\frac{d}{|x-y|}\right)= |x-y|-d <  |x-y+b(x)-b(y)| <|x-y|+d=|x-y| \left(1+\frac{d}{|x-y|}\right)
\]
and
\begin{align}
\label{eq:3}
(1/2)|x-y|<  |x-y+b(x)-b(y)| <(3/2)|x-y|.
\end{align}
Introducing polar coordinates we have
%We then consider the case $\lambda|x-y|\ge C_0+2\lambda |J|$. We have
\begin{align}
K_\lambda(x,y)=\int_0^\infty\e^{i\lambda^a(t(x)-t(y))r^a}\chi(r)^2\left( \int_{S^{n-1}} \e^{i\lambda r(x-y+b(x)-b(y))\cdot\xi^\prime}\,d\sigma(\xi^\prime)\right)r^{n-1}\,dr
\end{align} 
We observe that inner integral is $\hat \sigma(\lambda|x-y+b(x)-b(y)|r)$. According to Stein (\cite{Ste93}, p. 347) one has
\[ 
\hat\sigma(\xi)=c|\xi|^{1-n/2}J_{(n-2)/2}(|\xi|) .  
\]
We take $C_0$ large so that
\begin{align}
J_{(n-2)/2}(r)=&a_0\frac{\e^{ir}}{r^{1/2}}+a_1\frac{\e^{ir}}{r^{3/2}}+\dots+a_N\frac{\e^{ir}}{r^{N+1/2}}+\\
                       &b_0\frac{\e^{-ir}}{r^{1/2}}+b_1\frac{\e^{-ir}}{r^{3/2}}+\dots+b_N\frac{\e^{-ir}}{r^{N+1/2}}  + R(r)  \mbox{ for }r\ge C_0,
\end{align}
where $|R(r)|\lesssim\frac1{r^{N+3/2}} $ for $r\ge C_0$ (See \cite{Ste93} p. 338).\\
From this we get
\begin{align}
K_\lambda&(x,y)=\int_0^\infty\e^{i\lambda^a(t(x)-t(y))r^a}\chi(r)^2\,r^{n-1}\left(a_0\frac{\e^{i\lambda|x-y+b(x)-b(y)|r}}
															{(\lambda|x-y+b(x)-b(y)|r)^{\frac n2-\frac12}}+\right.\\ 
%&\left.+a_1\frac{\e^{i\lambda|x-y+b(x)-b(y)|r}}{(\lambda|x-y+b(x)-b(y)|r)^{\frac n2+\frac12}}+\dots}}\right.\\
&\left.\dots +b_N\frac{\e^{-i\lambda|x-y+b(x)-b(y)|r}}{(\lambda|x-y+b(x)-b(y)|r)^{N+\frac n2-\frac12}}
+ R_1(\lambda|x-y+b(x)-b(y)|r)\right)\,dr
\label{eq:4}
\end{align}
where $R_1(r)=r^{1-n/2}R(r)$.
 It follows from \eqref{eq:3} that\[
 R_1(\lambda |x-y|-b(x)-b(y)r)\lesssim \frac1{(\lambda|x-y+b(x)-b(y)|)^{N+n/2+1/2}}\le \frac1{(\lambda|x-y|/2)^{N+n/2+1/2}}  \\.
\]
The remainder term contribute with a remainder part of the kernel\[
K_{\lambda, \mbox{rem}}(x,y)\lesssim(\lambda|x-y|)^{-N-n/2-1/2}
\]
\par
Set $\Phi_\lambda(r)=\lambda^a(t(x) -t(y))r^a+\lambda|x-y+b(x)-b(y)|r$\\
The main term in (\eqref{eq:4}) gives the following contribution to $K_\lambda(x-y)$:
\begin{align}
K_{\lambda,0}(x,y)=\frac{a_0}{(\lambda|x-y+b(x)-b(y)|)^{n/2-1/2}}.\int_0^\infty \e^{i\Phi_\lambda(r)}\chi(r)^2r^{n-1}r^{\frac12-\frac n2}\,dr 
\end{align}
We consider two cases:\\
{\em Case 1:}\quad$|x-y|>>\lambda^{a-1}|t(x)-t(y)|$ gives $\Psi_\lambda^\prime\gtrsim\lambda|x-y|$ and
integrations by parts gives $|K_{\lambda,0}(x,y)|\lesssim(\lambda|x-y|)^{-N}$ for any large $N$.\\
{\em Case 2:} \quad $|x-y|\lesssim\lambda^{a-1}|t(x)-t(y)|$.
 In this case we have
 \[
 \Phi_\lambda^{\prime\prime}=\lambda^a\left(t(x)-t(y)\right)a(a-1)r^{a-2}
 \]
 and in the case $a\ne1$ van der Corput gives \[
 |K_{\lambda,0}(x,y)|\lesssim\lambda^{-a/2}|t(x)-t(y)|^{-1/2}\left(\lambda|x-y|\right)^{\frac12-\frac n2}
 \lesssim\left(\lambda|x-y| \right)^{-n/2}. \]
 In Case 2a with $a=1$ we have  the trivial estimate estimate \[ 
  |K_{\lambda,0}(x,y)|\lesssim\left(\lambda|x-y|\right)^{(1-n)/2}
 \]
 and the other terms in \eqref{eq:4} can be estimated in the same way.\\
 We note that  Case 2 is contained in the set $|x-y|\lesssim\lambda^{a-1}|J|$
 
 The other terms in \eqref{eq:1} can be estimated in the same way. \\
 
 To summarize the estimates we se that $|K_\lambda{(x,y)}|\lesssim G_\lambda(x-y)$ where\\
  when $a\ne1$: \[
G_\lambda(x)=\chi_{\{|x|<C_0\lambda^{-1}+2d\}}(x) +\chi_{\{|x|\ge \lambda^{-1}\}}\lambda^{-N}|x|^{-N} +\chi_{\{|x|\le C\lambda^{a-1}|J|\}}\lambda^{-n/2}|x|^{-n/2}
 \]
 and when $a=1$:
 \[
G_\lambda(x)=\chi_{\{|x|<C_0\lambda^{-1}+2d\}}(x) +\chi_{\{|x|\ge \lambda^{-1}\}}\lambda^{-N}|x|^{-N} +\chi_{\{|x|\le C|J|\}}\lambda^{(1-n)/2}|x|^{(1-n)/2}
 \]
 \par
 In the case when $a\ne1$ we have
  \[
  \|G\|_1\lesssim (\lambda^{-1}+d)^n+\lambda^{-n}+\int_{|x|\le C\lambda^{a-1}|J|}\lambda^{-n/2}|x|^{-n/2}\,dx
  \]
  and the above integral is majorized by\[
  \lambda^{-n/2}\int_0^{C\lambda^{a-1}|J|}r^{n/2-1}\,dr\lesssim\lambda^{-n/2}\lambda^{(a-1)n/2}|J|^{n/2}=
  \lambda^{\frac{n}2(a-2)}  |J|^{n/2}. \]
  Hence\[
  \|G\|_1\lesssim\lambda^{-n}+d^n+ \lambda^{\frac{n}2(a-2)}  |J|^{n/2}
  \]
  in the case $a\ne1$.
  In the case $a=1$ we get \[
   \|G\|_1\lesssim\lambda^{-n}+d^n+  \lambda^{(1-n)/2}\int_0^{C|J|}r^{1/2+n/2}\,dr\lesssim\lambda^{-n}+d^n+ \lambda^{(1-n)/2}  |J|^{(n+1)/2}.
 \]
  In the case $a\ne1$ we obtain \[
 \|T_\lambda T_\lambda^*\|\lesssim\lambda^{-n}+d^n+ \lambda^{\frac{n}2(a-2)}  |J|^{n/2},
 \]
 and\[
 \|T_\lambda\|\lesssim \lambda^{-n/2}+ d^{n/2}+\lambda^{\frac{n}4(a-2)}  |J|^{n/4}.
 \]
 From this we get
 \begin{align}
 \|S_\lambda f(x+b(x),t(x))\|_2&\le\lambda^n\|T_\lambda[\hat f(\lambda\cdot)]\|_2\le\lambda^n\|T_\lambda\|\cdot\|\hat f(\lambda\cdot)\|_2\\ 
 &\lesssim\lambda^n\left( \lambda^{-n/2}+  d^{n/2}+\lambda^{\frac{n}4(a-2)} |J|^{n/4} \right)\lambda^{-n/2}\|f\|_2\\
 &=\left(1+ d^{n/2}\lambda^{n/2}+\lambda^{na/4}|J|^{n/4}\right)\|f\|_2.
\end{align} 
\par
Finally in the case $a=1$ we obtain
 \[
 \|T_\lambda T_\lambda^*\|\lesssim\lambda^{-n}+d^n+\lambda^{(1-n/)2}  |J|^{(n+1)/2},
 \]
 \[
 \|T_\lambda\|\lesssim \lambda^{-n/2}+d^{n/2}+\lambda^{(1-n)/4}|J|^{(n+1)/4},
 \]
 and
 \begin{align}
 \|S_\lambda f(x+b(x),t(x))\|_2&\le\lambda^n\|T_\lambda[\hat f(\lambda\cdot)]\|_2\le\lambda^n\|T_\lambda\|\cdot\|\hat f(\lambda\cdot)\|_2\\ 
 &\lesssim\lambda^n\left( \lambda^{-n/2}+d^{n/2} +\lambda^{(1-n)/4)} |J|^{(n+1)/4} \right)\lambda^{-n/2}\|f\|_2\\
 &=\left(1+d^{n/2}\lambda^{n/2}+\lambda^{(n+1)/4}|J|^{(n+1)/4}\right)\|f\|_2.
\end{align} 
This completes the proof of Lemma 2.
\end{proof}
\par
We shall then extend Lemma 2.
\begin{lemma}
In Lemma 1 and Lemma 2 above the condition $\supp \hat f \subset \{\xi; \lambda/2\le|\xi|\le\lambda \}$ can be replaced
by the weaker condition  $\supp \hat f \subset \{\xi; |\xi|\le\lambda \}$
\end {lemma} 
In the proof of Lemma 3 we shall use a result in Sj\"olin and Str\"omberg \cite{Sjo-Str2}. Let $y_0\in\R^n, t_0\in\R^n,0<r\le1$, and let $f\in L^2(\R^n)$
with $\supp\hat f\subset B(0;\lambda)$ and $\lambda\ge1$. Set \[
F=\{ (y,t); y_{y_0,j}\le y_j\le y_{y_0,j}+r \mbox{ for }1\le j\le n,\mbox{ and }t_0\le t\le t_0+r^a\}.
\]
It is proved that \[
\|S_F^*f\|_2\lesssim(1+r^a\lambda^a)(1+r\lambda)^n\|f\|_2,
\]
and the method in \cite{Sjo-Str2} can be used to proof that
\begin{align}
\label{eq:5}
\|S_E^*f\|_2\lesssim(1+|J|\lambda^a)(1+r\lambda)^n 
\end{align}
if $E=B\times J$ with $B$ and $J$ as in Lemma 2.\\
The method to prove this is a generalisation of the method we used to prove Theorem A.
\begin{proof}[Proof of Lemma 3]
Let $N$ be the smallest integer such that $|J|2^{-aN}\lambda^a + r2^{-N}\lambda<2$.  We write $f=\sum_0^N f_j$  where 
$\supp\hat f_j\subset\{2^{-j-1}\lambda\le|\xi|\le2^{-j}\lambda\}$ for $0\le j<N$
 and $\supp\hat f_N\subset B(0;2^{-N}\lambda)$. 
 It follows from \eqref{eq:5} that\[ 
 \|S_E^*f_N\|_2\lesssim (1+|J|2^{-aN}\lambda^a)(1+r2^{-N}\lambda 
 )^n\|f_N\|_2 \lesssim\|f\|_2.\]
Also \[
\|S_E^*f\|_2\le\sum^N_{j=0}\|S_E^*f_j\|_2.
\]
and according to Lemma 2 we have for $a\ne1$
\[
 \|S_E^*f_j|\|_2\le C (2^{-jan/4} |J|^{n/4}\lambda^{na/4}+r^{n/2}\lambda^{n/2}2^{-jn/2}) \|f\|_2.
 \]
 for $0\le j<N$
 It follows that \[
 \left\|S^*_E \left(\sum_0^{N-1} f_j\right)\right \|_2\lesssim ( |J|^{n/4}\lambda^{na/4}+r^{n/2}\lambda^{n/2}) \|f\|_2.
 \]
 and we obtain\[
 \|S_E^*f\|_2\lesssim( |J|^{n/4}\lambda^{na/4}+r^{n/2}\lambda^{n/2}+1) \|f\|_2
 \]
 for $a\ne1$.\\
 The same proof works also for $a=1$ and this completes the proof of Lemma 3.
 \end{proof}
\par
We shall then prove the following theorem.
\begin{theorem}
Assume $a>0,n\ge1$ and $\lambda\ge1$. Let the interval  $J\subset[0,1]$, let $B$ be a ball in $\R^n$ with radius $r$ and
set $E=B\times J$.
   Let $f$ be an function in $L^2(\R^n)$ with $\supp \hat f\subset B(0;\lambda)$.\\
Then the following holds \\
when $n=1$ and $a\ne1$:
 \[
\|S_E^*f\|_2\le\left(|J|^{1/4}\lambda^{a/2}+r^{1/2}\lambda^{1/2}+1\right)\|f\|_2,
\]
when $n=1$ and $a=1$:
 \[
\|S_E^*f\|_2\le\left(|J|^{1/2}\lambda^{a/2}+r^{1/2}\lambda^{1/2}+1\right)\|f\|_2,
\]
when $n\ge2$ and $a\ne1$:
 \[
\|S_E^*f\|_2\le\left(|J|^{1/2}\lambda^{a/2}+r\lambda+1\right)(r\lambda+1)^{(n-2)/2}\|f\|_2,
\]
when $n\ge2$ and $a=1$:
 \[
\|S_E^*f\|_2\le\left(|J|^{1/2}\lambda^{1/2}+r^{n/(n+1)}\lambda^{n/(n+1)}+1\right)(r^{n/(n+1)}\lambda^{n/(n+1)}+1)^{(n-1)/2}\|f\|_2.
\]
\end{theorem}
\begin{proof}
The cases with $n=1$ in Theorem 3 follow directly from Lemma 3.\\
In the cases with $n\ge2$ we shall use an argument similar to the proof of Theorem 1 by covering the Interval $J$ with intervals $J_i$ of equal length.\\ 
In the case $a\ne1$ we have by Lemma 3 the estimate\[
 \|S_E^*f\|_2^2\le C\left( |J|^{n/2}\lambda^{na/2}+r^{n}\lambda^{n}+1 \right)\|f\|_2^2.
 \]
 We have \[  |J|^{n/2}\lambda^{na/2}+r^n\lambda^n+1\sim \left(|J|\lambda^{a}\right)^{n/2}+\left(r^2\lambda^2+1\right)^{n/2}.
\]
Cover $J$ with intervals $J_i, i=1,2,\dots,N$, of intervals of length $|J_i|$ such that \(
|J_i|\lambda^a=r^2\lambda^2+1\)  with $N\le |J|/|J_i|+1=|J|\lambda^a\left(r^2\lambda^2+1\right)^{-1}+1$
Set $E_i=B\times J_i$ then we have \[
 \|S_{E_i}^*f\|_2^2\lesssim\left(\left(|J_i|\lambda^{a}\right)^{n/2}+\left(r^2\lambda^2+1 \right)^{n/2}\right)\|f\|_2^2
 =2\left(r^2\lambda^2+1 \right)^{n/2}\|f\|_2^2
\]
and
\begin{align}
\|S_E^*f\|_2^2&\le\sum\limits_{i=1}^N \|S_{E_i}^*f\|_2^2\lesssim N\left( r^2\lambda^2+1 \right)^{n/2}\|f\|_2^2\\
&\le \left(|J|\lambda^a\left(r^2\lambda^2+1\right)^{-1}+1\right)\left( r^2\lambda^2+1\right)^{n/2}\|\|f\|_2^2\\
&=\left(|J|\lambda^a+r^2\lambda^2+1\right)\left(r^2\lambda^2+1\ \right)^{(n-2)/2}\|f\|_2^2\\
&\le\left(|J|^{1/2}\lambda^{a/2}+r\lambda+1\right)^2\left(r^{(n-2)/2}\lambda^{(n-2)/2}+1\right)^2\|f\|_2^2,
\end{align}
which gives the desired estimate in this case.\\
In the case $a=1$ we have by Lemma 3 the estimate
\[
 \|S_E^*f\|_2^2\le C\left(  |J|^{(n+1)/2}\lambda^{(n+1)a/2}+r^n\lambda^n+1 \right)\|f\|_2^2.
\]
We have \[  |J|^{(n+1)/2}\lambda^{(n+1)a/2}+r^n\lambda^n+1
\sim \left(|J|\lambda^{a}\right)^{(n+1)/2}+\left(r^{2n/(n+1)}\lambda^{2n/(n+1)}+1\right)^{(n+1)/2}.
\]
Cover $J$ with intervals $J_i, i=1,2,\dots,N$, of intervals of length $|J_i|$ such that \(
|J_i|\lambda^a=r^{2n/(n+1)}\lambda^{2n/(n+1)}+1\)  and $N\le |J|/|J_i|+1=|J|\lambda^a\left(r^{2n/(n+1)}\lambda^{2n/(n+1)}+1 \right)^{-1}+1$
Set $E_i=B\times J_i$ then we have 
\begin{align}
 \|S_{E_i}^*f\|_2^2&\lesssim\left(\left(|J_i|\lambda^{a}\right)^{(n+1)/2}+\left( (r^{2n/(n+1)}\lambda^{2n/(n+1)}+1 \right)^{(n+1)/2}\right)\|f\|_2^2\\
 &=2\left( r^{2n/(n+1)}\lambda^{2n/(n+1)}+1  \right)^{(n+1)/2}\|f\|_2^2,
\end{align}
and
\begin{align}
\|S_E^*f\|_2^2&\le\sum\limits_{i=1}^N \|S_{E_i}^*f\|_2^2\lesssim N\left( (r^{2n/(n+1)}\lambda^{2n/(n+1)}+1  \right)^{(n+1)/2}  \|f\|_2^2\\
&\le \left(|J|\lambda^a\left(r^{2n/(n+1)}\lambda^{2n/(n+1)}+1\right)^{-1}+1\right)\left(r^{2n/(n+1)}\lambda^{2n/(n+1)}+1\right)^{(n+1)/2}
\|\|f\|_2^2\\
&=\left(|J|\lambda^a+r^{2n/(n+1)}\lambda^{2n/(n+1)}+1\right)\left(r^{2n/(n+1)}\lambda^{2n/(n+1)}+1\ \right)^{(n-1)/2}\|f\|_2^2\\
&\le\left(|J|^{1/2}\lambda^{a/2}+r^{n/(n+1)}\lambda^{n/(n+1)}+1\right)^2\left(r^{n/(n+1)}\lambda^{n/(n+1)}+1\right)^{n-1}\|f\|_2^2,
\end{align}
which gives the desired estimate in this case.\\
\end{proof}
\section{A convergence result}
We shall here prove a convergence result for function in $H_s(\R^n), n\ge2$, and begin with two lemmas.
\begin{lemma}
 Assume  $a>0, a\ne1,n\ge2, \lambda\ge1$ and $0<s<a/2$. Also let  $(t_m)_1^\infty \in l^{r,\infty}$, where
 $r=2s/(a-2s)$ % Let , n\ge1$ and let $n_0=\min
 Then one has\[
 \|\sup\limits_m|S_{t_m} f|\|_2\lesssim\lambda^s\|f\|_2, 
 \]
 if $f\in L^2(\R^n)$ and $\supp\hat f\subset\{\xi;\lambda/2\le|\xi|\le\lambda\}$.
\end{lemma}
\begin{proof}
Let  $0<b<1$. We have \[
  \|\sup\limits_m |S_{t_m}f|\|_2\le\|\sup\limits_{t_m\le b} |S_{t_m}f|\|_2+\|\sup\limits_{t_m> b}|S_{t_m}f|\|_2=T_1+T_2.
  \]
 Theorem 1 gives the estimate\[
  T_1\lesssim\left(b^{1/2}\lambda^{a/2}+1 \right)\|f\|_2.
  \]
  We also have  $\#\{m;t_m>b\}\lesssim b^{-r}$ and it follows that 
  \[
  T_2\lesssim b^{-r/2}\|f\|_2.
  \]
  We then choose $b$ such that \[
  b^{1/2}\lambda^{a/2}=b^{-r/2}.
  \]
  One gets \[
  b^{1/2+r/2}=\lambda^{-a/2},
  \]
  and\[
  b=\lambda^{-a/(1+r})<1.
  \]
  Hence $b^{-r/2}\ge1$ and \[
  T_1\lesssim b^{1/2+r/2}\|f\|_2=b^{-r/2}\|f\|_2.
  \] 
  We have shown that\[
   \|\sup\limits_m |S_{t_m}f|\|_2\lesssim b^{-r/2}\|f\|_2,
   \]
   where\[
    b^{-r/2}=\lambda^{ra/2(1+r)}=\lambda^s.\]
 This completes the proof of the lemma.
\end{proof}
We shall then improve Lemma 4 by proving the following lemma
\begin{lemma} 
Let $a,n,s,r$, and $(t_m)_1^\infty $ have the same properties as in Lemma 4. and assume that $f\in H_s(\R^n)$. Then
\[
 \|\sup\limits_mS_{t_m} f|\|_2\lesssim\|f\|_{H_s} \mbox{ for } f\in\mathscr{S}(\R^n).
 \]
\end{lemma}
Before proving Lemma 5 we remark that the following theorem follows from Lemma 5 (see proof of Corollary 4 in Sj\"olin and Str\"omberg
\cite{Sjo-Str1})
\begin{theorem}
Let $a,n,s,r$, and $(t_m)_1^\infty $ have the same properties as in Lemma 4 and assume that $f\in H_s(\R^n)$. Then
\[
\lim\limits_{m\to\infty}S_{t_m}f(x)=f(x)\mbox{ for almost every }x. 
\]
\end{theorem}
\begin{proof}[Proof of Lemma 5]
It follows from the theorem on monotone convergence that instead of estimating $\sup\limits_m|S_{t_m}f|$ it is sufficient to estimate 
 $\sup\limits_{m\le M}|S_{t_m}f|$for large integers $M$ (as long as the estimates do not dependend on $M)$. We can find a 
 measurable function $t(x)$ such that \[
 \sup\limits_{m\le M}|S_{t_m}f|=|S_{t(x)}f(x)|,
 \] 
 and $t(x)$ takes only finitely many values. We then define intervals $I_k=(2^{-k-1},2^{-k}],\\ k=0,1,\dots,$ and sets 
 $F_k=\{x\in\R^n, t(x)\in I_k \}$. The Sets $F_k$ are disjoint and $\R^n=\bigcup\limits{k\ge0}F_k$. We let $X_k$
 denote the characteristic function of $F_k$ and then have $\sum\limits_{k\ge0}\chi_k=1$ and 
 \begin{align}
 \label{eq:6}
 S_{t(x)}f(x)=\sum\limits_{k\ge0}X_k(x)S_{t(x)}f(x).
 \end{align} 
We write a function $f=\sum_{j\ge0}f_j$ by splitting its Fourier transform $\hat f=\sum_{j\ge0}\hat f_j$\,, \\
where $\hat f_j$ i supported in
$\Omega_j$ where $\Omega_0=\{ |\xi|\le 1\}$ and  $\Omega_j=\{ 2^{-j-1}|\xi|\le 2^j\}$ for $j>0$.\\
We shall then split the sum \eqref{eq:6} into three parts. For $j\ge0$ set
\begin{align}
k(j)&=(a-2s)j=j2s/r,&b(j)&=2^{-k(j)},\\
b_1(j)&=2^{-k(j)-\epsilon_1j}\mbox{ \hspace{1.5cm}and }&b_2(j)&=2^{-k(j)+\epsilon_2j},
\end{align}
where $\epsilon_1=2\epsilon$ and $\epsilon_2=2\epsilon/r$ and $\epsilon$ is a small positive number.\\

We have\[
S_{t(x)}f(x)=\sum\limits_{k\ge0}\sum\limits_{j\ge0} X_k(x)S^a_{t(x)}fj(x)=S^1_{t(x)}f(x)+S^2_{t(x)}f(x)+S^3_{t(x)}f(x),
\] 
where
\begin{align}
S^1_{t(x)}f(x)&=\sum\limits_{j\ge0}
\sum\limits_{\substack{k\ge0 ,\\k\ge k(j)+\epsilon_1j}}
X_k(x)S_{t(x)}f_j(x),\\
S^2_{t(x)}f(x)&=\sum\limits_{j\ge0}\sum\limits_{\substack{k\ge0\\k\le k(j)-\epsilon_2j}}
X_k(x)S_{t(x)}f_j(x),\\
S^3_{t(x)}f(x)&=\sum\limits_{j\ge0}\sum\limits_{\substack{k\ge0\\k(j)-\epsilon_2j< k<k(j)+\epsilon_1 j}}
 X_k(x)S_{t(x)}f_j(x).
\end{align}
Invoking Theorem 1 we obtain
\begin{align}
|S^1_{t(x)}f(x)|&\le\sum\limits_{j\ge0}\sum\limits_{\substack{k\ge0\\k\ge k(j)+\epsilon_1j}} X_k(x)|S_{t(x)}f_j(x)|
\le\sum\limits_{j\ge0}\sup\limits_{t_m\le b_1(j)}|S_{t_m}f_j(x)|,
\end{align}
and\[
\|S^1_{t(x)}f\|_2\le \sum\limits_{j\ge0}\| \sup\limits_{t_m\le b_1(j)}|S_{t_m}f_j|\|_2\lesssim \sum\limits_{j\ge0}(1+b_1(j)^{1/2}2^{aj/2})\|f_j\|_2.
\]
We have\[
b_1(j)^{1/2}2^{aj/2}=2^{-\frac12k(j)-\frac12\epsilon_1j}2^{aj/2}=2^{-\frac12(a-2s+\epsilon_1-a)j}=2^{j(s-\epsilon)}, 
\]
 and it follows that
\[
\|S^1_{t(x)}f\|_2\lesssim \sum\limits_{j\ge0}2^{j(s-\epsilon)}\|f_j\|_2\le \|f\|_{H_s}.
\]
We shall then estimate $S^2_{t(x)}f(x)$. One has
\begin{align}
|S^2_{t(x)}f(x)|&\le|\sum\limits_{j\ge0}\sum\limits_{\substack{k\ge0\\k\le k(j)-\epsilon_2j}} X_k(x)|S_{t(x)}f_j(x)|
\le\sum\limits_{j\ge0}\sup\limits_{t_m\ge b_2(j)} |S_{t_m}f_j(x)|,
\end{align}
and we obtain\[
\|\sup\limits_{t_m\ge b_2(j)}|S_ {t_m}f_j(x)|\|^2_2\lesssim\sum\limits_{t_m>b_2(j)/2}\|f_j\|_2^2=\#\{m;t_m>b_2(j)/2\} \|f_j\|^2_2
 \lesssim b_2(j)^{-r}\|f_j\|^2_2 
\]
We also have \[b_2(j)^{-r}=2^{rk(j)-\epsilon_2rj},
\]
and\[
rk(j)-\epsilon_2rj= 2j(s-\epsilon).\]
It follows that
\[
\|S^2_{t(x)}f\|_2\lesssim \sum\limits_{j\ge0}2^{j(s-\epsilon)}\|f_j\|_2\le \|f\|_{H_s}.
\]
It remains to study $S^3_{t(x)}f(x)$. We let $[k(j)]$ denote the integral part of $k(j)$.
and setting $l=k-[k(j)]$ we obtain
\begin{align}
|S^3_{t(x)}f(x)|\le\sum\limits_{j\ge0}\sum\limits_{\substack{ k\ge0\\k(j)-\epsilon_2j<k<k(j)+\epsilon_1j}} X_k(x)|S_{t(x)}f_j(x)|\\
=\sum\limits_{l=-\infty}^\infty \sum\limits_{ j>\max\{ (l-1)/\epsilon_1,-l/\epsilon_2\}}X_{[k(j)]+l}(x)|S_{t(x)}f_j(x)|.
\end{align}
Using the fact that $X_k=X_k^2$ and applying Cauchy-Schwarz inequality one obtains
\begin{align}
&\left(\sum\limits_{ j>\max\{ (l-1)/\epsilon_1,-l/\epsilon_2\}}X_{[k(j)]+l}(x)|S_{t(x)}f_j(x)|\right)^2\\
 &\hspace{.5cm}\lesssim \left(\sum\limits_{ j>\max\{ (l-1)/\epsilon_1,-l/\epsilon_2\}}X^2_{[k(j)]+l}(x)\right)\left( \sum\limits_{ j>\max\{ (l-1)/\epsilon_1,-l/\epsilon_2\}}
 X_{[k(j)]+l}(x) |S_{t(x)}f_j(x)|^2\right).\end{align}
The first sum on the second line is majorized by \[
C_0\max_k \#\{ j; [k(j)]=k\}\lesssim1,
\]
and it follows that
\begin{align}
&\|\sum\limits_{ j>\max\{ (l-1)/\epsilon_1,-l/\epsilon_2\}}X_{k(j)+l}(x)|S_{t(x)}f_j(x)|\|_2^2
\lesssim \sum\limits_{ j>\max\{ (l-1)/\epsilon_1,-l/\epsilon_2\}}
\int X_{[k(j)]+l}|S_{t(x)}f_j||^2\,dx\\
&\hspace{4cm}\lesssim \sum\limits_{ j>\max\{ (l-1)/\epsilon_1,-l/\epsilon_2\}}\int\sup\limits_{t_m\in I_{[k(j)]+l}}
|S_t(x)f_j(x)|^2\,dx.
\end{align}
Invoking Minkovski's inequality we then obtain\[
\|S^3_{t(x)}f\|_2\lesssim=\sum\limits_{l=-\infty}^\infty \left( \sum\limits_{ j>\max\{ (l-1)/\epsilon_1,-l/\epsilon_2\}}
\|\sup\limits_{t_m\in I_{[k(j)]+l}}|S_{t_m}f_j|\|^2_2 \right)^{1/2}.
\]
Furthermore for $l\ge0$ we have by Theorem 1.
\begin{align}
\|\sup\limits_{t_m\in I_{[k(j)]+l}}|S_{t_m}f_j|\|^2_2&\lesssim(1+2^{[-k(j)]-l}2^{aj})\|f_j\|_2^2\lesssim(1+2^{-k(j)-l+aj})\|f_j\|_2^2\\
&\lesssim(1+2^{2sj}2^{-l})\|f_j\|_2^2\lesssim2^{-l} 2^{2sj}\|f_j\|_2^2.
\end{align}
For $l\le0$ we have
\begin{align}
\|\sup\limits_{t_m\in I_{[k(j)]+l}}|S_{t_m}f_j|\|^2_2&\lesssim(\#\{m;t_m\in I_{[k(j)]+l}\})\|f_j\|_2^2\lesssim2^{-r(-[k(j)]-l)}\|f_j\|_2^2
\lesssim2^{rl}2^{2sj}\|f_j\|_2^2.
\end{align}
We conclude that
\begin{align}
\|S^3_{t(x)}f\|_2\lesssim\sum\limits_{l=-\infty}^0\left(\sum\limits_{ j>-l/\epsilon_2} 2^{rl}2^{2sj}\|f_j\|_2^2 \right)^{1/2}
   +\sum\limits_{l=1}^\infty \left( \sum\limits_{ j>(l-1)/\epsilon_1}2^{-l}2^{sj}\|f_j\|_2^2\right)^{1/2}\\
\le\left(\sum\limits_{l=-\infty}^0 2^{rl/2}+\sum\limits_{l=1}^\infty2^{-l/2}\right)\left( \sum\limits_{j>0}2^{2sj}\|f_j\|^2_2\right)^{1/2}\lesssim\|f\|_{H_s}.
\end{align}  
This completes the proof of  Lemma 5.
\end{proof}
\section{Relations between maximal estimates in one variable and maximal estimates in dimension $n\ge2$.}
Next we shall consider the Schr\"odinger equation on radial or symmetric functions on $\R^n$ and will see how it can be reduced to a one-dimensional problem\\
\par
{\em Remark } In this paper we have the Fourier transform  $\hat f$ of a function on $\R^n$ defined by \eqref{eq:Fouriertransform},
and then yields $\|\hat f\|_{L^(\R^n)}=(2\pi)^{n/2} \| f\|_{L^(\R^n)}$. We set  \[
\alpha_n=(2\pi)^{n/2},\mbox{ for } \,n\ge1
\]
in this section.\\
\par
Let $S_t^{(k)}$ denote the $k$-dimensional Schr\"odinger operator (with a given $a>0$ in its definition) and let\[
S_E^{*(k)}f(x)=\sup\limits_{t\in E}|S_t^{(k)}f(x)|
\]
where $f$ is a function on $\R^k$ and the supremum is taken  over a set $E\subset[0,1]$.\\ 
\par
{\em Remark.} We may in this section replace the Fourier multiplier functions $\{\e^{it|\xi|^s}\}_t$ with any family of radial Fourier multiplier functions 
$\{\hat k_t(|\xi|)\}_t$ satifying $|\hat k_t(|\xi|)|\le1$.\\
\par
We shall prove the following theorem.
\begin{theorem}
Let $s\ge0$, let $n\ge2$ and let  $E$ be a given subset of the interval $[0,1]$-
\par
If \[
\|S^{*(1)}_Ef\|_2 \le C \|f\|_{H_s}
\]  
for all functions $f$ in $\mathscr{S}(\R)$,\\
\par
then\[
\|S^{*(n)}_Ef\|_2 \le C_{n,k} \|f\|_{H_s}
\]
for all functions $f$  in  $\mathscr{S}(\R^n)$ of the form $f(x)=f_0(x)P(x)$, where $f_0\in\mathscr{S}(\R^n)$ and is radial
and $P$ is a solid spherical harmonic on $\R^n$ of degree $k\ge0$ \\[.2cm]
\end{theorem}
Theorem 5 will follow with some approximation arguments from the following theorems.
\begin{theorem}
Let $f_1$ be a function  $L^2(\R)$, supported on $\R_+$, not identically zero; and equal to a function in $\mathscr{S}(\R)$ on $[1,\infty)$.
 Let $P$ be a solid spherical harmonic
on $\R^n$ of degree $k$ and
 normalised so that $\|P\|_{L^2(\S^{n-1})}=1$.\\
 Let $f_P$ be the symmetric function on $\R^n$ defined by its
Fourier transform \[
\hat f_P(\xi)=P(\xi') f_1(|\xi|)|\xi|^{1/2-n/2}  \mbox{ for } \xi=\xi'|\xi|\in\R^n.
\]
Let $\check f_1$ be the inverse Fourier transform of $f_1$ on $\R$. Note that $\alpha_1 \|\check f\|_{L^2(\R)}=\|f_1\|_{L^2(\R+)}$.
Assume that $E$ is  subset of the interval $[0,1]$. 
Then there is a constant $C_{n,k}$  dependent only on $n,k$  but independent of $a>0, E$ and $f_1$ such that
\begin{align}
\label{eq:SEn_left}
\alpha_n\|S_E^{*(n)}f_P\|_{L^2(\R^n)}\le \alpha_1\sqrt{2} \|S_E^{*(1)} \check f_1\|_{L^2(\R)} +C_{n,k}\|f_1\|_{L^2(\R+)}.
\end{align}
\end{theorem}
{\em Remark.} Most of the results in this section can also be formulated with norms on spheres,  where the $L^2(\R^n)$
is replaced by the norms $L^2(\S^{n-1}(r))$ for $r>0$. \\
For instance \eqref{eq:SEn_left} may be replaced by\[
\alpha_n\|S_E^{*(n)}f_P\|_{L^2(\S^{n-1}(r))}\le \alpha_1\left(|S_E^{*(1)} \check f_1(r)|^2+|S_E^{*(1)} \check f_1(-r)|^2\right)^{1/2} +Tf_1(r), \mbox{ for }r>0,
\]
where $\|Tf_1\|_{L^2(\R+)} \le C_{n,k}\|f_1\|_{L^2(\R+)}$.\\[.4cm]
With  two solid spherical  harmonic the following estimates  holds.
\begin{theorem} Let the functions  $f_1, P ,f_P$  and the set $E$ be as in Theorem 6 and let $Q$ be a solid spherical harmonic
on $\R^{n_1}$ of degree $k_1$ and
 normalised so that $\|Q\|_{L^2(\S^{n_1-1})}=1$ and define $f_Q$ by its Fourier transform
  \[
\hat f_Q(\xi)=Q(\xi')f_1(|\xi|)|\xi|^{1/2-n_1/2}  \mbox{ for } \xi=\xi'|\xi|\in\R^{n_1}.
\]
Set  $v=k+n/2-1$ and $v_1=k_1+n_1/2-1$.\\
If $v=v_1$ then\[
\alpha_n\|S_E^{*(n)}f_P\|_{L^2(\R^n)} = \alpha_{n_1} \|S_E^{*(n_1)} f_Q\|_{L^2(\R^{n_1})}.
\] 
It  $2v_1= 2v\mod(4)$ then\[
 \left |\alpha_n\|S_E^{*(n)}f_P\|_{L^2(\R^n)} -  \alpha_{n_1} \|S_E^{*(n_1)} f_Q\|_{L^2(\R^{n_1})}\right| \le(C_{n,k}+C_{n_1,k_1})\|f_1\|_{L^2(\R+)}.
\]
\end{theorem}
{\em Remark.} In the special case $n=1$ we have $k=0,1$ and use  a special definition of $f_P$. 
When $k=0$ the even function $f_e=f_P$ is defined by\[
   \hat f_e(\xi)=\left\{\begin{array}{ll}\frac1{\sqrt2}f_1(\xi)\mbox{ for }\xi\ge0,\\
   						\frac1{\sqrt2}f_1(-\xi)\mbox{ for }\xi<0,\end{array}\right.	
   \]
   and when  $k=1$ the odd function $f_o=f_P$ is defined by\[
   \hat f_o(\xi)=\left\{\begin{array}{ll}\hspace{.3cm}\frac1{\sqrt2}f_1(\xi)\mbox{ for }\xi\ge0,\\
   						-\frac1{\sqrt2}f_1(-\xi)\mbox{ for }\xi<0.\end{array}\right.	
\]
\par
Estimates in the opposite  direction of \eqref{eq:SEn_left}   in Theorem 6 are somewhat more complicated.\\
\par
We will  state the key estimate in this section. First let us define the complex unit vectors\[
\gamma(v)=\e^{-i[\frac{\pi v}2 +\frac\pi4]}.
\]
We have the following
\begin{proposition}
Let $f$ be a function in $L^2(\R)$  whose Fourier transform $\hat f$ is equal to a function in $\mathscr{S}(\R)$ 
 on $(-\infty,-1]$ and on $[1,\infty)$\\ 
Let $n\ge1,k\ge0, v=n/2+k-1$ and $\gamma(v)$ as above. Assume that $\hat f$ satisfies the symmetry \[
 \gamma(v) \hat f(-r)=\overline{\gamma}(v)\hat f(r),\mbox{ for all }r>0.
\]
Let $P$ be a solid spherical harmonic
on  $\R^n$ of degree $k$  normalised so that $\|P\|_{L^2(\S^{n-1})}=1$. Let $f_P$ be the symmetric function on $\R^n$ defined by its
Fourier transform \[
\hat f_P(\xi)=P(\xi')\hat f(|\xi|)|\xi|^{1/2-n/2}  \mbox{ for } \xi=\xi'|\xi|\in\R^n.
\]
Let  $E$ be any  subset of the interval $[0,])$ containing $0$ . 
Then there is a constant $C_v$   dependent only on $v$  but independent of $a>0, E$ and $f$ such that\[
\alpha_n\|S_E^{*(n)}f_P\|_{L^2(\R^n)} =\alpha_1 \left(\int_0^\infty |S_E^{*(1)}\ f(x)|^2\,dx\right)^{1/2} +R(f,v)
\]
where the restterm \[
|R(f,v)|\le C_v\|f\|_{L^2(\R)}.
\]
The restterm $R(f,v)$ depend on the parameter $a>0$ in the definition of $S_t$ and the set $E$.
\end{proposition} 
Proposition 1 follows directly from Proposition 2 with Corollary 3  and Proposition 3 which are stated and proved at the end of this section. 
\par
Using Proposition 1 it is easy to prove Theorems 6 and 7
\begin{proof}[Proofs of Theorem 6 and Theorem 7]
When  $2v_1= 2v\mod(4)$ then $\gamma(v_1)=\pm\gamma{v}$. Define the function $f$ by its Fourier transform
\begin{align}c
\label{eq:f1-extended}
\hat f(\xi)=\left\{\begin{array}{l}f_1(\xi), \xi>0,\\
                                     ( \overline{\gamma}(v))^2f_1(-\xi)=(\overline{\gamma}(v_1))^2f_1(-\xi), \xi<0.
                                        \end{array}               \right.
\end{align}
By Proposition 1
\[
\alpha_n\|S_E^{*(n)}f_P\|_{L^2(\R^n)} =\alpha_1 \left(\int_0^\infty |S_E^{*(1)}\ f(x)|^2\,dx\right)^{1/2} +R(f,v)
\]
\[
\alpha_n\|S_E^{*(n_1)}f_P\|_{L^2(\R^{n_1})} =\alpha_1 \left(\int_0^\infty |S_E^{*(1)}\ f(x)|^2\,dx\right)^{1/2} +R(f,v_1)
\]
 we conclude that
 \begin{align}
\left|\alpha_n\|S_E^{*(n)}f_P\|_{L^2(\R^n)} -\alpha_n\|S_E^{*(n_1)}f_P\|_{L^2(\R^{n_1})}\right|=|R(f,v)-R(f,v_1)|\\
\le|R(f,v)|+|R(f,v_1)|\le (C_v+C_{v_1})\|f_1\|_{L^2(\R+)}.
\end{align}
In the special case $v=v_1$ we get \[
\alpha_n\|S_E^{*(n)}f_P\|_{L^2(\R^n)} =\alpha_n\|S_E^{*(n_1)}f_P\|_{L^2(\R^{n_1)}}.
\]
This completes the proof of Theorem 7.\\
For the proof of Theorem 6 we let $f$  be defined by \eqref{eq:f1-extended}. We will use following property of the Fourier transform on $L^2(\R^n)$ 
Let  the operator $g \to g_-$ is defined by $g_-(x)=g(-x)$. This operator commutes with the Fourier transform and also with the inverse Fourier
tranform. We have
\begin{align}
(g_-)\hat{}=(\hat g)_-= \alpha_n\check g\mbox{ and }(g_-)\check{}=(\check g)_-= (\alpha_n)^{-1}\hat g.
\end{align}
From this we get  \[ 
|S_t^{(1)}f(x)|^2=| S_t^{(1)}\check f_1(x)+\overline{\gamma}^2S_t^{(1)}\check f_1(-x)|^2+ \le 2(|S_t^{(1)}\check f_1(x)|^2 +|S_t^{(1)}\check f(-x)|^2) \mbox{ for } x>0
\] 
and hence\[
|S_E^{*(1)}f(x)|^2 \le 2(|S_E^{*(1)}\check f_1(x)|^2+ |S_E^{*(1)}\check f_1(-x)|^2) \mbox{ for } x>0.
\]
Integrating over the positive interval we get \[
\int_0^\infty |S_E^{*(1)}\ f(x)|^2\,dx\le2 \|S_E^{*(1)}\check f_1\|^2_{L^2(\R)} .
\]
From this and Proposition 1 we get the desired estimate in Theorem 6
\end{proof}
Now we consider estimates in the opposite direction of the estimate in Theorem 6. First a lemma which
follows directly from Proposition 1
\begin{lemma}
Let $f$ be a function  in $L^2(\R)$  equal to a function in $\mathscr{S}(\R)$ on $(-\infty,-1]$ and on $[1,\infty)$  .
Let $\hat f$ be the Fourier  transform of $f$. 
Let $f_P$ be the symmetric function on $\R^n$ defined by its
Fourier transform \[
\hat f_P(\xi)=P(\xi')\hat f(|\xi|)|\xi|^{1/2-n/2}  \mbox{ for } \xi=\xi'|\xi|\in\R^n.
\]
If $\hat f$  satisfy the symmetry $\gamma(v) \hat f(-r)=\overline{\gamma}(v)\hat f(r)$ for all $r>0$, then 
 \[
\alpha_1\left(\int^\infty_0|S_E^{*(1)} f(x)|^2\,dx\right)^{1/2}\le   \alpha_n \|S_E^{*(n)}f_P\|_{L^2(\R^n)} + C_{n,k}\|f\|^2_{L^2(\R)}. 
\] 
\end{lemma}
\begin{proof}[Proof of Lemma 6]
From Proposition 1 we get    \[
 \alpha_1 \left(\int_0^\infty |S_E^{*(1)}\ f(x)|^2\,dx\right)^{1/2}=\alpha_n\|S_E^{*(n)}f_P\|_{L^2(\R^n)} -R(f,v)
\]
with the restterm satisfying  $|-R(f,v )|\le C_v\|f\|_{L^2(R_+)}$ . This completes the proof of the lemma.
\end{proof}
\par
Next we will see how we can combine several estimates like those in Lemma 7. We have
\begin{lemma}
Let $f$ be a function  in $L^2(\R)$  equal to a function in $\mathscr{S}(\R)$ on $(-\infty,-1]$ and on $[1,\infty)$  .
Let $\hat f$ be the Fourier  transform of $f$. 
Let $n\ge1,k\ge0$ and $n_1\ge1,k_1\ge0$  and let $P$ be a solid spherical harmonic
on $\R^n$ of degree $k$ and and  $Q$ be a solid spherical harmonic
on $\R^{n_1}$ of degree $k_1$ 
and  normalised so that $\|P\|_{L^2(\S^{n-1})}=\|Q\|_{L^2(\S^{n_1-1})}=1$.\\
 Let $ \boldsymbol{f}_P$ be the symmetric function on $\R^n$ defined by its
Fourier transform \[
\boldsymbol{\hat f}_P(\xi)=P(\xi')\left(-\overline{\gamma}(v_1)\hat f(|\xi|)+\gamma(v_1)\hat f(-|\xi|)\right)|\xi|^{1/2-n/2}  \mbox{ for } \xi=\xi'|\xi|\in\R^n,
\]
and let $\boldsymbol {f}_Q$ be the symmetric function on $\R^{n_1}$ defined by its
Fourier transform \[
\boldsymbol{\hat f}_Q(\xi)=Q(\xi')\left(\overline{\gamma(v)}\hat f(|\xi|)-\gamma(v)\hat f(-|\xi|)\right)|\xi|^{1/2-n_1/2}  \mbox{ for } \xi=\xi'|\xi|\in\R^{n_1},
\] 
where  $v=k+n/2-1$ and $v_1=k_1+n_1/2-1$ and assume that $2v_1\ne 2v\mod(4)$, i.e.
 $\gamma(v_1)\ne \gamma(v)$\\
Let $E$ be any subset of the interval $[0,1)$. 
Then there is a constant $C_{n,k}$   and a constant $C_{n_1,k_1}$ not dependent of $a>0, E$ and $f$ such that
\begin{align}
c\alpha_1\left(\int_0^\infty|S_E^{*(1)}f(x)|^2\,dx\right)^{1/2}&\le \alpha_n \|S_E^{*(n)}\boldsymbol{f}_P\|_{L^2(\R^n)} +
 \alpha_{n_1}\|S_E^{*(n_1)}\boldsymbol{f}_Q\|_{L^2(\R^{n_1})}\\
&+(C_{n,k}+ C_{n_1,k_1})\|f\|_{L^2(\R)} .
\end{align}
where $c=2|\sin\frac{\pi(v-v_1)}{2}|\in\{2,\sqrt2\}$.
\end{lemma}
As a special case of Lemma 6 we have
\begin{corollary} 
Assume $f$ satisfy the symmetry
\begin{align}
\label{eq:symmetry}
\hat f(\xi)=\left\{\begin{array}{ll}c_1g(\xi) &\mbox{ for } \xi>0,\\
                          c_2 g(-\xi) &\mbox{ for } \xi<0,
                          \end{array}\right.\mbox{ with } (c_1,c_2)\in \C^2, (c_1,c_2)\ne0,
\end{align}
%\eqref{eq:symmetry}. \\
and set $c=2|\sin\frac{\pi(v-v_1)}{2}|\in\{2,\sqrt2\}$,  $c'_1=c_1\overline{\gamma(v_1)}-c_2\gamma(v_1)$ and $c'_2=c_1\overline{\gamma(v)}-c_2\gamma(v)$.\\
 Then
\begin{align}
c\alpha_1\left(\int^\infty_0|S_E^{*(1)} f(x)|^2\,dx\right)^{1/2}\le  & |c'_1|\alpha_n\|S_E^{*(n)}f_P\|_{L^2(\R^n)}+ |c'_2|\alpha_{n_1} \|S_E^{*(n_1)}f_Q\|_{L^2(\R^{n_1})}\\ 
&+ (C_{n,k}+C_{n_1,k_1})\|f\|_{L^2(\R)} 
\end{align}
where
 \begin{align}
\hat f_P(\xi)=P(\xi')g(|\xi|) |\xi|^{1/2-n/2}   \mbox{ for } \xi=\xi'|\xi|\in\R^n,
\end{align}
and 
\begin{align}
\hat f_Q(\xi)=Q(\xi')g(|\xi|)|\xi|^{1/2-n/2}   \mbox{ for } \xi=\xi'|x|\in\R^{n_1}.
\end{align}
\end{corollary}
We are mostly interested in the cases when $(c_1,c_2)$ is $(1,1)$, $(1,-1)$, $(1,0)$ or $(0,1)$ which corresponds to even or odd functions
or functions with Fourier transform supported on the positive or negative half-axis.
 \begin{proof}[Proof of Corollary 1]
 The functions $\boldsymbol{f}_P$ and $\boldsymbol{f}_Q$ defined in the Lemma 7 have Fourier transforms
\begin{align}
\label{eq:f_P}
 \boldsymbol{\hat f}_P(\xi)=c_1'P(\xi')g(|\xi|) |\xi|^{1/2-n/2}   \mbox{ for } \xi=\xi'|\xi|\in\R^n,
\end{align}
and 
\begin{align}
\label{eq:f_Q}
 \boldsymbol{\hat f}_Q(\xi)=c'_2Q(\xi')g(|\xi|)|\xi|^{1/2-n/2}   \mbox{ for } \xi=\xi'|x|\in\R^{n_1},
\end{align}
for   $c'_1=c_1\overline{\gamma(v_1)}-c_2\gamma(v_1)$  and $c'_2=c_1\overline{\gamma(v)}-c_2\gamma(v)$. The desired estimates
follows directly from Lemma 7.
 \end{proof}
\par
\begin{comment}
 From Corollary 2 and Corollary 1 we  conclude the following.
\begin{corollary}
Assume $f$ satisfy the symmetry 
\begin{align}
\hat f(\xi)=\left\{\begin{array}{ll}c_1g(\xi) &\mbox{ for } \xi>0,\\
                          c_2 g(-\xi) &\mbox{ for } \xi<0,
                          \end{array}\right.\mbox{ with } (c_1,c_2)\in \C^2, (c_1,c_2)\ne0,
\end{align}
and set $c=2|\sin\frac{\pi(v-v_1)}{2}|\in\{2,\sqrt2\} , c_1''=(| c_1\overline{\gamma}(v_1)-c_2\gamma(v_1)|^2+
 | c_2\overline{\gamma}(v_1)-c_1\gamma(v_1)|^2)^{1/2}$ and $c_2''=(| c_1\overline{\gamma}(v)-c_2\gamma(v)|^2+
 | c_2\overline{\gamma}(v)-c_1\gamma(v)|^2)^{1/2}$.\\
%\end{align}\eqref{eq:symmetry}. \\
Then
\begin{align}
c\alpha_1 \| S_E^{*(1)} f \|_{L^2(\R)}\le   &c_1''\alpha_n \|S_E^{*(n)}f_P\|_{L^2(\R^n)}+ c_2''\alpha_{n_1} \|S_E^{*(n_1)}f_Q\|_{L^2(\R^{n_1})} \\
&+ (C_{n,k}+C_{n_1,k_1})\|f\|_{L^2(\R)}
 \end{align}
 where
 \begin{align}
\label{eq:f_P}
\hat f_P(\xi)=P(\xi')g(|\xi|) |\xi|^{1/2-n/2}   \mbox{ for } \xi=\xi'|\xi|\in\R^n,
\end{align}
and 
\begin{align}
\label{eq:f_Q}
\hat f_Q(\xi)=Q(\xi')g(|\xi|)|\xi|^{1/2-n/2}   \mbox{ for } \xi=\xi'|x|\in\R^{n_1},
\end{align}
Especially, when $f$ is even and  $\overline{\gamma}(v)/\gamma(v)=1$;\quad
 or when  $f$ is odd and  $\overline{\gamma}(v)/\gamma(v)=-1$,
the following holds.
\[
  \alpha_1 \| S_E^{*(1)} f \|_{L^2(\R)}\le   \alpha_n\sqrt2  \|S_E^{*(n)}f_P\|_{L^2(\R^n)}+ C_{n,k}\|f\|_{L^2(\R)}.
   \]
\end{corollary}
\end{comment}
 \begin{proof}[Proof of Lemma 7] We will essentially only do elementary calculation in $\C^2$,  and use the triangle inequality
for $L^2$ norms. Some details are left to the reader.\\
  Set\[
 \boldsymbol{f}_{P,1}(r)=-\overline{\gamma}(v_1)\hat f(r)+\gamma(v_1)\hat f(-r) \mbox{ for }r>0,
 \]
 and\[
 \boldsymbol{f}_{Q,1}(r)=\overline{\gamma}(v)\hat f(r)-\gamma(v)\hat f(-r)\mbox{ for }r>0.
 \]
 Then
\begin{align}
\| \sup\limits_{t\in E}\tilde S_t\boldsymbol{f}_{P,1}\|_{L^2(\R_+)}=\alpha_n\|S_E^{*(n)}\boldsymbol{f}_P\|_{L^2(\R^n)}
\end{align}
 and
 \begin{align}
\| \sup\limits_{t\in E}\tilde S_t\boldsymbol{f}_{Q,1}\|_{L^2(\R_+)}=\alpha_{n_1}\|S_E^{*(n_1)}\boldsymbol{f}_Q\|_{L^2(\R^{n_1})},
\end{align}
where $\tilde S_t$ is define as in \eqref{eq:defStilde}.
Let the functions $f_{P,2}$ and $f_{Q,2}$ on $\R$ be defined by \[
\boldsymbol{f}_{P,2}(r)=\left\{\begin{array}{ll} \gamma(v) f_{P,1}(r) &\mbox{ for }r\ge0,\\
                                          \overline{\gamma}(v) f_{P,1}(-r) &\mbox{ for }r<0,
                                          \end{array}\right.
                                          \]
   and
    \[
\boldsymbol{f}_{Q,2}(r)=\left\{\begin{array}{ll} \gamma(v_1) f_{Q,1}(r) &\mbox{ for }r\ge0,\\
                                          \overline{\gamma}(v_1) f_{Q,1}(-r) &\mbox{ for }r<0.
                                          \end{array}\right.
                                          \]
    We have
    \begin{align}
    \boldsymbol{f}_{P,2}+\boldsymbol{f}_{Q,2}&=\left\{\begin{array}{l} (-\overline{\gamma}(v_1)\gamma(v) +\overline{\gamma}(v)\gamma(v_1)) f(r)
    + (\gamma(v_1)\gamma(v)-\gamma(v)\gamma(v_1)) f(-r)\mbox{ for }r\ge0,\\
                             (-\overline{\gamma}(v_1)\overline{\gamma}(v)+ \overline{\gamma}(v)\overline{\gamma}(v_1) )f(r)
                              +(\gamma(v_1)\overline{\gamma}(v)-\gamma(v)\overline{\gamma}(v_1) )f(-r)        \mbox{ for }r<0,
                                          \end{array}\right.\\
                              &=(\gamma(v_1)\overline{\gamma}(v)-\gamma(v)\overline{\gamma}(v_1) )\hat f(r) \mbox{ for } r\in\R.           
                                              \end{align}
   
  By the assumption  $2v_1\ne 2v\mod(4)$,   we have   
  \begin{align}
  |\gamma(v_1)\overline{\gamma}(v)-\gamma(v)\overline{\gamma}(v_1)| 
  =|\e^{-i[\frac{\pi v_1}2 +\frac\pi4]}\e^{i[\frac{\pi v}2 +\frac\pi4]}-\e^{-i[\frac{\pi v}2 +\frac\pi4]}\e^{i[\frac{\pi v_1}2 +\frac\pi4]}|\\
  =|\e^{-i\frac\pi2(v_1-v)}-\e^{i\frac\pi2(v_1-v)}|=|2\sin(\pi(v_1-v)/2)|=c\in\{2,\sqrt2\}
  \end{align}
   %$\gamma(v)=\e^{i[\frac{\pi v}2 -\frac\pi4]}$
   %\ne0§                                                                                  
 Lemma 7 now follows from Lemma 6 and Proposition 3.\\
 \end{proof}
 Next we observe that  $S_t^{(1)}f(-x)=S_t^{(1)}f_-(x)$, where $\hat f_-(\xi)=\hat f(-\xi)$.  Thus  if $f$ is an even or an odd function
on $L^2(\R^n)$ then   $S_E^{*(1)} f$ is even and we obtain \[
 \| S_E^{*(1)} f \|_{L^2(\R)}=\sqrt{2}\left(\int^\infty_0|S_E^{*(1)} f(x)|^2\,dx\right)^{1/2}.
\]
We also observe that any function $f$ on $\R$ can be written as a  sum of an even and an odd function $f= f_e+f_o$.
Then also the Fourier transforms $\hat f_e$ and $\hat f_o$ are even respective  odd and we have $\hat f=\hat f_e+\hat f_o$
and that $\|f\|^2_2=\|f_e\|^2_2+\|f_o\|^2_2$.\\
%\end{comment}
Using Corollary 1 of Lemma 7 we  get the following estimates in different cases.
\begin{theorem}
Let $f$ be a function  in $L^2(\R)$  equal to a function in $\mathscr{S}(\R)$ on $(-\infty,-1]$ and on $[1,\infty)$  .
Decompose $f$ into even and odd functions $f=f_e+f_o$ with Fourier transforms $\hat f_e$  and $\hat f_o$  
Let $n\ge1,k\ge0$ and $n_1\ge1,k_1\ge0$  and let $P$ be a solid spherical harmonic
on $\R^n$ of degree $k$ and and  $Q$ be a solid spherical harmonic
on $\R^{n_1}$ of degree $k_1$ 
and  normalised so that $\|P\|_{L^2(\S^{n-1})}=\|Q\|_{L^2(\S^{n_1-1})}=1$. Define the symmetric functions
\begin{align}
\hat f_{e,P}(\xi)&=\hat f_e(|\xi|)P(\xi')|\xi|^{1/2-n/2} \mbox{ for } \xi=\xi'|\xi| \in\R^n,\\
\hat f_{e,Q}(\xi)&=\hat f_e(|\xi|)Q(\xi')|\xi|^{1/2-n_1/2} \mbox{ for } \xi=\xi'|\xi| \in\R^{n_1},\\
\hat f_{o,P}(\xi)&=\hat f_o(|\xi|)P(\xi')|\xi|^{1/2-n/2} \mbox{ for } \xi=\xi'|\xi| \in\R^n,\\
\hat f_{o,Q}(\xi)&=\hat f_o(|\xi|)Q(\xi')|\xi|^{1/2-n_1/2} \mbox{ for } \xi=\xi'|\xi| \in\R^{n_1}.
\end{align}
%and so on.\\
Let $\gamma(v)=\e^{-i[\frac{\pi v}2 +\frac\pi4]}$, $v=k+n/2-1$ and $v_1=k_1+n_1/2-1$  and
assume that $2v_1\ne 2v\mod(4)$.\\
Then we have in different cases.\\
\par
Case 1. 
Assume $\overline{\gamma}(v)/\gamma(v)=1$ and $\overline{\gamma}(v_1)/\gamma(v_1)=-1$.\\
Then
\begin{align}
\alpha_1\|S_E^{*(1)}f_e\|_{L^2(\R)} &\le   \alpha_n\sqrt2 \|S_E^{*(n)}f_{e,P}\|_{L^2(\R^n)} +C'_{n,k}\|f_e\|  ,\\
\alpha_1\|S_E^{*(1)}f_o\|_{L^2(\R)} &\le  \alpha_{n_1}\sqrt2\|S_E^{*(n_1)}f_{o,Q}\|_{L^2(\R^{n_1})} +C'_{n_1,k_1}\|f_o\|_2,
\end{align}
and
\begin{align}
\alpha_1\|S_E^{*(1)}f\|_{L^2(\R)}  \le  \alpha_n\sqrt2\|S_E^{*(n)}f_{e,P}\|_{L^2(\R^n)} &+  \alpha_{n_1}\sqrt2 \|S_E^{*(n_1)}f_{o,Q}\|_{L^2(\R^{n_1})}\\
&+(C'_{n,k}+C'_{n_1,k_1})\|f\|_{L^2(\R)}.
\end{align}
\par Case 2. 
Assume that  $\overline{\gamma}(v)/\gamma(v)=i$ and $\overline{\gamma}(v_1)/\gamma(v_1)=-i$. \\
Then 
\begin{align}
\alpha_1\|S_E^{*(1)}f_e\|_{L^2(\R)} \le   \alpha_n\|S_E^{*(n)}f_{e,P}\|_{L^2(\R^n)} + &\alpha_{n_1}\|S_E^{*({n_1})}f_{e,Q}\|_{L^2(\R^{n_1})}\\
 &+(C'_{n,k}+C'_{n_1,k_1})\|f_{e}\|_{L^2(\R)},\\  
\alpha_1\|S_E^{*(1)}f_o\|_{L^2(\R)} \le   \alpha_n \|S_E^{*({n})}f_{o,P}\|_{L^2(\R^n)} &+ \alpha_{n_1}\|S_E^{*(n_1)}f_{o,Q}\|_{L^2(\R^{n_1})}\\
&+(C'_{n,k}+C'_{n_1,k_1})\|f_o\|_{L^2(\R)},\\  
\end{align} 
and
\begin{align}
 \alpha_1\|S_E^{*(1)}f\|_{L^2(\R)} &\le  \alpha_n \|S_E^{*(n)}f_{e,P}\|_{L^2(\R^n)} +\alpha_{n_1} \|S_E^{*({n_1})}f_{e,Q}\|_{L^2(\R^{n_1})}\\
  &+\alpha_n\|S_E^{*({n})}f_{o,P}\|_{L^2(\R^n)} +\alpha_{n_1}\|S_E^{*(n_1)}f_{o,Q}\|_{L^2(\R^{n_1})}\\
  &+(C'_{n,k}+C'_{n_1,k_1})\|f\|_{L^2(\R)}.
\end{align}
\par
Case 3:
Assume  $\overline{\gamma}(v)/\gamma(v)=1$ and  $\overline{\gamma}(v_1)/\gamma(v_1)=i$.\\
Then
\begin{align}
\|S_E^{*(1)}f_e\|_{L^2(\R)} &\le   \alpha_n\sqrt2  \|S_E^{*(n)}f_{e,P}\|_{L^2(\R^n)} +C_{n,k}\|f_e\|_{L^2(\R)}  ,\\
\|S_E^{*(1)}f_o\|_{L^2(\R)} &\le  \alpha_n \sqrt2\|S_E^{*({n})}f_{o,P}\|_{L^2(\R^n)} +2\alpha_{n_1} \|S_E^{*(n_1)}f_{o,Q}\|_{L^2(\R^{n_1})}\\
&+(C'_{n,k}+C'_{n_1,k_1})\|f_o\|_{L^2(\R)}
\end{align}
and
\begin{align}
 \|S_E^{*(1)}f\|_{L^2(\R)} &\le   \alpha_n\sqrt2  \|S_E^{*(n)}f_{e,P}\|_{L^2(\R^n)}  +\alpha_n \sqrt2\|S_E^{*({n})}f_{o,P}\|_{L^2(\R^n)} \\
  &+2\alpha_{n_1}\|S_E^{*(n_1)}f_{o,Q}\|_{L^2(\R^{n_1})} +(C'_{n,k}+C'_{n_1,k_1})\|f\|_{L^2(\R)}.
\end{align}
We have similar statements with  $\overline{\gamma}(v)/\gamma(v)= -1$  and / or   $\overline{\gamma}(v_1)/\gamma(v_1)=-i$ 
\end{theorem}
\begin{comment}
\[
\frac1{C_{n,k,k_1,n_1}} \le \frac{\|S_E^{*(n_1)}f_Q\|_{L^2(\R^{n_1})}}{ \|S_E^{*(n)}f_P\|_{L^2(\R^n)}}\le C_{n,k,n_1,k_1}
\]
and\[
\frac1{C_{n,k}}  \le \frac{|S_E^{^(1)}\check f_1\|_{L^2(\R)}}{ \|S_E^{*(n)}f_P\|_{L^2(\R^n)}}\le C_{n,k}.
\]
\end{comment}
It remains to prove Proposition 1. For this we will use Proposition 2 and Proposition 3 below.\\
First we define the operator $\tilde S_t$ for $t>0$ by
\begin{align}
\label{eq:defStilde}
\tilde S_t g(r)=\int_0^\infty J_v(rs)(rs)^{1/2}g(s)\e^{its^a}\,ds
\end{align}
for $g\in L^2(\R_+)$. 
Here $J_m$  is the Bessel function of order $m$ with integer or half-integer  $m>-1$.
Then we have
\begin{proposition}
Let $f_1$ be a function  in $L^2(\R_+)$ .
Let $P$ be a solid spherical harmonic
on  $\R^n$ of degree $k$  normalised so that $\|P\|_{L^2(\S^{n-1})}=1$.
Let $\hat f$ be the Fourier  transform of $f$ and
Let $f_P$ be the symmetric function on $\R^n$ defined by its
Fourier transform \[
\hat f_P(\xi)= f_1(|\xi|)P(\xi')|\xi|^{1/2-n/2}  \mbox{ for } \xi=\xi'|\xi|\in\R^n
\]
Let $\tilde S_t$ defined as in \eqref{eq:defStilde} with $v=n/2+k -1$. 
Then
\begin{align}
S_t^{(n)}f_P(x)=c_{k,n}\alpha_n^{-1}|x|^{1/2-n/2}\tilde S_tf_1(|x|)P(-x')\mbox{ for } x=x'|x|\in\R^n,
\end{align}
with  $|c_{k,n}|=1$.
\end{proposition}
We get the corollaries.
\begin{corollary}
$\tilde S_t$ has norm $1$ on $L^2(\R^n)$.
\end{corollary}
\begin{corollary}We have\[
S^{*(n)}_E f_P(x)r=\alpha_n^{-1}r^{1/2-n/2}\tilde S^*_Ef_1(r)|P(-x')|\mbox{ for } x=rx'\in\R^n
\]
and \[
\alpha_n\|S^{*(n)}_E f_P\|_{L^2(\R^n)}=\|\tilde S^*_E f_1\|_{L^2(\R_+)}.
\]
\end{corollary}
\begin{proof}[Proof of Corollary 2]
We have \[
\|\tilde S_t f_1\|_{L^2(\R_+)}=\alpha_n\|S_t^{(n)}f_P\|_{L^2(\R^n)}=\|\e^{it|\xi|^a}\hat f_P\|_{L^2(\R^n)}
=\|\hat f_P\|_{L^2(\R^n)}=\| f_1\|_{L^2(\R_+)}.
\]
\end{proof}
\begin{proof}[Proof of Corollary 3] First part follows directly, as we may take out the factor $\alpha_n^{-1}r^{1/2-n/2}|P(-x')|$,
from the the supremum on the left hand side. 
Second part is obtained by integration over $\R^n$.
\end{proof}
We will use  the following estimate for the operator $\tilde S_t$
\begin{proposition}
Let $\tilde S_t$ be defined as in \eqref{eq:defStilde} with $v=n/2+k -1$ the and let $f$ be a function in $L^2(R)$
Let $f_1$ be a function in $L^2(\R+)$ and  let $f$ be the function in $L^2(\R)$ with
with Fourier transform \[
\hat f(\xi)=\left\{\begin{array}{l}\gamma(v)f_1(\xi)\mbox{ for }\xi>0,\\
						\overline{\gamma}(v)f_1(-\xi)\mbox{ for }\xi<0.\end{array}\right.						
\]
Then \[ 
\tilde S_t f_1(x)= \alpha_1 S^{(1)}_tf(x)+ R_{t,v}f_1(x)
\]
where the remainder terms satisfies \[
\left(\int_0^\infty \sup\limits_{t\in E} |R_{t,v}f_1(x)|^2\, dx\right)^{1/2}\le C_v \|f_1\|_{L^2(\R_+)}.
\]
\end{proposition}
It remains to prove the propositions in the section. Proposition 1 follows directly from Proposition 2 with Corollary 3  and Proposition 3 
\begin{proof}[Proof of Proposition 1]
Let $f_1(\xi)=\overline{\gamma(v)}\hat f(\xi)$ for $\xi>0$.
We have \[
\alpha^n\|S^{*(n)}_E f_P\|_{L^2(\R^n)}=\|\tilde S^*_E f_1\|_{L^2(\R_+)}\le\alpha_1\|S^{*(1)}_E f\|_{L^2(\R_+)}+\||R_{t,v}f_1\|_{L^2(\R_+)}
\]
Set $R(f,v)=\int_0^\infty \sup\limits_{t\in E} |R_{t,v}f_1(x)|^2\, dx$ Then
\[ 
R(f,v)=
\left(\int_0^\infty \sup\limits_{t\in E} |R_{t,v}f_1(x)|^2\, dx\right)^{1/2}\le C_v \|f_1\|_{L^2(\R_+)}=\frac{\alpha_1}{\sqrt(2)}C_v\|f\|_{L^2(\R)}.
\]
and we obtain the desired estimate.
This completes the proof of Proposition 1.
\end{proof}
\begin{proof}[Proof of Proposition 2]
Assume $f=f_0P$ as in the above theorem. It follows from Stein and Weiss \cite{Ste-Wei}, p. 158  that (when  
$\hat f(x)=F_0(x)P(x)$ where
\begin{align}
\label{eq:SteWei}
F_0(x)=c_{n,k}\alpha_n r^{1-n/2-k}\int_0^\infty f_0(s)J_{n/2+k-1}(rs)s^{n/2+k}\,ds
\end{align}
for $r=|x|$. Here $J_m$ denotes Bessel functions of order $m$. and $c_{n,k}$ is a constant with
 $|c_{n,k}|=1$. 
 (The one-dimensional cases follow elementary as \\
 %\begin{align}
 $
J_{-1/2}(r))=\sqrt{2}\cos(r)/\sqrt{\pi r}$ and $ J_{1/2}(r))=\sqrt{2}\sin(r)/\sqrt{\pi r}.\hspace{.2cm})
$\\
%\end{align} \\
Also\[
S_tf(x)=\alpha_n^{-1} c_{n,k}r^{1-n/2-k}\left(\int_0^\infty J_{n/2+k-1}(rs)F_0(s)\e^{its^a}s^{n/2+k}\,ds\right)P(-x),
\] 
where $r=|x|>0$. We set \[
f_1(|\xi|)=F_0(|\xi|)|\xi|^{-1/2+n/2+k},
\]
then $\|f_1\|_{L^2(\R+)}=\|\hat f_p\|_{L^2(\R^n)} $.\\ 
We obtain
\begin{align}
\label{eq:tilde}
S_tf_P(x)&=\alpha_n^{-1}c_{n,k}r^{1/2-n/2}\left(\int_0^\infty J_{n/2+k-1}(rs)f_1(s)(rs)^{1/2}\e^{its^a}\,ds\right)P(-x/r|)\\
&=\alpha_n^{-1}c_{n,k}r^{1/2-n/2}\tilde S_tf_1(r)P(-x/r),
\end{align}  
where the operator $\tilde S_t$ is defined by \eqref{eq:defStilde}.\\
%\begin{align}\tilde S_tf_1(r)=\int_0^\infty J_{n/2/+k-1}(rs)(rs)^{1/2}f_1(s)\e^{its^a}\,ds.\end{align}
This completes the proof of Proposition 2.
\end{proof}
\par
For proof of Proposition 3 we need the following
\begin{lemma}
Let $K(r)$ be a non-negative function on $(0,\infty)$ satisfying \[
\int_0^\infty\frac {K(r)}{\sqrt{r}}\,dr=A<\infty,
\]
and let \[
Tf(s)=\int_0^\infty K(rs)f(r)\,dr
\]
for $f\in L^2(0,\infty)$.\\
Then we have \[
\|Tf\|_{L^2(\R_+)}\le A\|f\|_{L^2(\R_+)},
\]
\end{lemma}
\begin{proof}[Proof of Proposition 3]
 We have the Bessel function\[
J_v(r)=\sqrt{\frac2{\pi r}} \cos(r-\pi v/2-\pi/4)+ \mathscr{O}(r^{-3/2})\mbox{ as } r\to\infty.
\]
See Stein and Weiss \cite{Ste-Wei}. p. 158. It follows that\\
\begin{align}
\label{eq:r_halfJ}
r^{1/2}J_v(r)=\gamma_v\e^{ir}+\overline{\gamma}_v\e^{-ir}+K_v(r)
\end{align}
for $r>0$, where 
 \[
\gamma_v=\frac1{\sqrt{2\pi}}\e^{-i(\pi v/2+\pi /4)}=\alpha_1\gamma(v)\frac1{2\pi}
\]
and
\[
 |K_v(r)|\le C_v \frac1{1+r}.
\]
We get 
\begin{align}
         \tilde S_tf_1(r)&=\gamma(v)\alpha_1\frac1{2\pi}\int_0^\infty \e^{irs}\e^{its^a} f_1(s)\,ds +\overline{\gamma}(v)\alpha_1\frac1{2\pi}\int_0^\infty \e^{-irs}\e^{its^a} f_1(s)\,ds\\
         &\hspace{1cm}+\int_0^\infty K_v(rs)\e^{its^a} f_1(s)\,ds\\
&=\alpha_1\frac1{2\pi}\int_0^\infty \e^{irs}\e^{it|s|^a}\gamma(v) f_1(s)\,ds +\alpha_1\frac1{2\pi}\int_{-\infty}^0 
         \e^{irs}\e^{it|s|^a}\overline{\gamma}(v)f_1(-s)\,ds+R_{t,v}f_1(r)\\ 
       &=\alpha_1\frac1{2\pi}\int_{-\infty}^\infty \e^{irs}\e^{it|s|^a}\hat f(s)\,ds           
        =\alpha_1S_t^{(1)} f(r)
        +R_{v,t}f_1(r)
         \end{align} 
for $r>0$.\\
Let $\tilde R_v$ be the sublinear operator on $L^2[0,\infty)$ defined by
\begin{align}
\label{eq:sublinearT}
\tilde R_vg(r)=\int_0^\infty | K_v(rs)|\,|g(s)|\,ds. 
\end{align}
Then\[
\sup\limits_{t\ge0}|R_{v,t}f_1(r)|\le\tilde R_vf_1(r).
\]
Since $\int_0^\infty |K_v(r)|r^{-1/2}\,dr\le C_v\int_0^\infty r^{-1/2}(1+r)^{-1/2}\,dr<\infty$ we obtain
by Lemma 8
\[
\left(\int_0^\infty \sup\limits_{t>0}|R_{v,t}\hat f_1(r)|^2\right)^{1/2}\le \left(\int_0^\infty (\tilde R_v\hat f_1(r))^2\,dr\right)^{1/2}\le C_v \|f_1\|_2.
\]
which completes the proof of the Proposition 3.
\end{proof}
\par
In this section it remains only to prove Lemma 8.\\
\begin{proof}[Proof of Lemma 8]
We have\[
Tf(s)=\int_0^\infty K(u)f(\frac us)\frac1s\,du, s>0,
\]
and we set \[
h(s)=f\left(\frac1s\right)\frac1s \hspace{.5cm}\mbox{ and }\hspace{.5cm} h_r(s)=h\left(\frac sr\right)\frac1{\sqrt{r}}\hspace{.7cm}\mbox{ for }r>0\mbox{ and }s>0.
\]
It then follows that\[
f\left(\frac rs\right)\frac1s=f\left(\frac1{s/r}\right)\frac1{s/r}\frac1r=h\left(\frac sr\right)\frac1r=h_r(s)\frac1{\sqrt{r}}.
\]
We have\[
Tf(s)=\int_0^\infty K(r)h_r(s)\frac1{\sqrt{r}}\,dr,
\]
and an application of Minkovski's inequality gives
\[
\|Tf\|_2\le\int_0^\infty K(r)\frac1{\sqrt{r}}\,\|h_r\|_2\,dr.
\]
We observe that \[
\|h_r\|_{L^2(\R_+)}=\|h\|_2=\|f\|_{L^2(\R_+)},
\]
and get\[
\|Tf\|_{L^2(\R_+)}\le\int_0^\infty K(r)\frac1{\sqrt{r}}\,dr\, \|f\|_{L^2(\R_+)}=A\|f\|_{L^2(\R_+)},
\] 
This completes the proof of Lemma 8.
\end{proof}
\section{A counter-example}
We shall give a counter-example in dimension $n\ge2$.
\begin{theorem}
Assume that $(t_k)_1^\infty $ is decreasing and  $(t_k-t_{k+1})_1^\infty $ is decreasing and $\lim\limits_{k\to\infty}t_k=0$. Assume 
$a>0, a\ne1$, and $0<s<a/4$, and $n\ge2$, Set $r(s)=2s/(a-4s)$ and assume that  $(t_k)_1^\infty\notin l^{r(s),\infty} $.\\
Then there is no estimate\[
\|\sup\limits_k|S_{t_k}f|\|_2\lesssim\|f\|_{H_s}
\]
for all radial function $f\in\mathscr{S}(\R^n)$.
\end{theorem} 
\begin{proof}
In the case $n=1$ this theorem i proved in Dimou and Seeger \cite{Dim-See} (with the exception that their counter-example is not radial) and
we shall modify their proof. Assuming that  $(t_k)_1^\infty \notin l^{r(s),\infty}$ Dimou and Seeger first construct sequences $(b_j)_1^\infty$
and $(M_j)_1^\infty$ of positive numbers such that $\lim\limits_{j\to\infty}b_j=0$ and $\lim\limits_{j\to\infty}M_j=\infty$. Taking 
$\epsilon<10^{-1}(a+2)^{-1}$ they then set \[
 \lambda_j=M_j^{\frac2a}b_j^{-\frac1{a-4s}}\mbox{ and }\rho_j=\epsilon b_j^{-1/2}\lambda_j^{1-a/2} =\epsilon M_j^{\frac{2-a}a}b_j^{-\frac{1-2s}{a-4s}}
\]
for $j=1,2,3,\dots$. We shall consider these numbers for $j$ large and observe that $\rho_j/\lambda_j=$\\
$\epsilon M_j^{-1}b^{\frac{2s}{a-4s}}\le\epsilon$.\\
In \cite{} the function\[
\Phi_{\lambda,\rho}(\xi,x,t)=x(\rho\xi-\lambda)+t(\lambda-\rho\xi)^a,|\xi|\le1/2,x\in\R,t>0,
\]
is studied an it is proved that for $x\in I_j=[0,a\lambda_j^{a-1}b_j/2]$ there exists $t_{k(x,j)}$ such that 
\begin{align}
\label{eq:7}
\max\limits_{|\xi|\le1/2}\left|\e^{i\Phi_{\lambda_j,\xi_j} (\xi,x,t_{k(x,j)})}-1\right|\le1/2.
\end{align}
We shall use the inequality \eqref{eq:7} in our proof and shall also use that
\begin{align}
\label{eq:8}
\rho_j\lambda_j^{a-1}b_j\to\infty\mbox{ as }j\to\infty
\end{align}
To prove \eqref{eq:8} observe that
\begin{align}
\rho_j\lambda_j^{a-1}b_j&=\epsilon M_j^{\frac{2-a}a}b_j^{-\frac{1-2s}{a-4s}}M_j^{\frac{2(a-1)}a}b_j^{-\frac{a-1}{a-4s}}b_j\\
&=\epsilon M_j^{\frac{2-a+2a-2}a}b_j^{\frac{-1+2s-a+1+a-4s}{a-4s}}=\epsilon M_jb_j^{-\frac{2s}{a-4s}}
\end{align}
which implies \eqref{eq:8}.\\
Then set $J_j=[a\lambda_j^{a-1}b_j/4,a\lambda_j^{a-1}b_j/2]$ and let $C_1$ be a large constant. It follows from \eqref{eq:8} that 
$\lambda_j\lambda_j^{a-1}b_j\to\infty\mbox{ as }j\to\infty$ and hence\[
2C_1\le a\lambda_j\lambda_j^{a-1}b_j/4
\]
and
\[
\frac{2C_1}{\lambda_j}\le a\lambda_j^{a-1}b_j/4
\]
for large $j$. We conclude that 
\begin{align}
\label{eq:9}
|x|\in J_j\mbox{ implies }\lambda_j|x|\ge2C_1.
\end{align} 
Now let $\sigma$ denote the surface measure on the unit sphere in $\R^n$. We have
\begin{align}
\label{eq:10}
\hat\sigma(y)=c_1\frac{\e^{i|y|}}{|y|^{n/2-1/2}} +c_2\frac{\e^{-i|y|}}{|y|^{n/2-1/2}} +R(y),
\end{align}
where
\begin{align}
\label{eq:11}
|R(y)|\lesssim\frac1{|y|^{n/2+1/2}}\le\frac\delta{|y|^{n/2-1/2}}\mbox{ for }|y|\ge C_1
\end{align}
and $\delta$ is small. (See Stein \cite{Ste93} , p. 347).\\
Then assume that $g\in C_0^\infty(\R)$, $\supp g\subset[-1/2,1/2]$, $g\ge0$, $\int g\,dx=1$, and $g$ even. We define a function
 $f\in\mathscr{S}(\R^n)$ by setting\[
 \hat f(\xi)=\frac1\rho\, g\left(\frac{|\xi|-\lambda}\rho\right)\mbox{ for }\xi\in\R^n.
 \] 
 Here $\lambda=\lambda_j$ , $\rho=\rho_j$ and $f=f_j$.\\
 It is easy to see that $\hat f(\xi)\ne0$ implies $\lambda-\rho/2\le|\xi|\le\lambda+\rho/2$. We also have\[
 \int|\hat f(\xi)|^2\lambda^{2s}\, d\xi\lesssim\int\limits_{\lambda-\rho/2}^{\lambda+\rho/2}\rho^{-2}\lambda^{2s}r^{n-1}\, dr\lesssim\rho^{-1}\lambda^{2s+n-1}
 \]
 and
 \begin{align}
 \label{eq:12}
 \|f\|_{H_s}\lesssim\rho^{-1/2}\lambda^{s+n/2-1/2}
 \end{align}
 Using polar coordinates we have
 \begin{align}
 S_tf(x)&=c \int \e^{i\xi\cdot x} \e^{it|\xi|^a}\hat f(\xi)\, d\xi=c\int \e^{i\xi\cdot x} \e^{it|\xi|^a}\frac1\rho\, g\left(\frac{|\xi|-\lambda}\rho\right)\,d\xi\\
 &=c\int_0^\infty\frac1\rho\, g\left(\frac{r-\lambda}\rho\right)\e^{itr^a}\left(\,\,\int\limits_{S^{n-1}}\e^{ir\xi^\prime\cdot x}\,d\sigma(\xi^\prime)\right)r^{n-1}dr\\
 &=c\int_0^\infty\frac1\rho\, g\left(\frac{r-\lambda}\rho\right)\e^{itr^a}\hat\sigma(rx)r^{n-1}dr
 \end{align}
 where by \eqref{eq:10}\[
 \hat\sigma(rx)=c_1\frac{\e^{ir|x|}}{(r|x|)^{n/2-1/2}} +c_2\frac{\e^{-ir|x|}}{(r|x|)^{n/2-1/2}} +R(rx).
 \]
 We assume $|x|\in J_j$ and \eqref{eq:9} gives $\lambda|x|\ge2C_1$ and $r|x|\ge C_1$ in the above integral. Hence by \eqref{eq:11} \[
 |R(rx)|\le\delta(r|x|)^{-n/2+1/2}.
 \] 
 It follows that
\begin{align}
S_tf(x)&=c_1c\int_0^\infty\frac1\rho\, g\left(\frac{r-\lambda}\rho\right)\e^{itr^a}\frac{\e^{ir|x|}}{(r|x|)^{n/2-1/2}}r^{n-1}\,dr\\
&+c_2c\int_0^\infty\frac1\rho\, g\left(\frac{r-\lambda}\rho\right)\e^{itr^a}\frac{\e^{-ir|x|}}{(r|x|)^{n/2-1/2}}r^{n-1}\,dr+c\int_0^\infty\frac1\rho\, g\left(\frac{r-\lambda}\rho\right)\e^{itr^a}R(rx)r^{n-1}\,dr
 \end{align}
 Setting $\xi=(r-\lambda)/\rho$ so that $r=\lambda+\rho\xi$ we obtain\[
 S_tf(x)=c_1S^1_tf(x)+c_2S^2_tf(x)+S^3_tf(x)
 \]
 where\[
 S^1_tf(x)=\left(c\int g(\xi)\e^{i[t(\lambda+\rho\xi)^a+|x|(\lambda+\rho\xi)]}(\lambda+\rho\xi)^{n/2-1/2}\,d\xi\right)\, |x|^{1/2-n/2},    
 \]
 \[
 S^2_tf(x)=\left(c\int g(\xi)\e^{i[t(\lambda+\rho\xi)^a-|x|(\lambda+\rho\xi)]}(\lambda+\rho\xi)^{n/2-1/2}\,d\xi\right)\, |x|^{1/2-n/2},
 \]
 and
 \[
 |S^3_tf(x)|\lesssim\delta\int(\lambda+\rho\xi)^{n/2-1/2}g(\xi)\,d\xi\,|x|^{1/2-n/2}\le C\delta\lambda^{n/2-1/2}|x|^{1/2-n/2}
 \]
 In $S^1_tf(x)$ and $S^2_tf(x)$ we can replace $\xi$ by $-\xi$. 
   We use that $g$ is even  and get the phase functions\[
  \Phi_1(\xi)=|x|(\lambda-\rho\xi)]+t(\lambda-\rho\xi)^a
  \]
  and
  \[
  \Phi_2(\xi)=-|x|(\lambda-\rho\xi)]+t(\lambda-\rho\xi)^a
  \]
  and  replace $(\lambda+\rho\xi)^{n/2+1/2}$ by $(\lambda-\rho\xi)^{n/2-1/2}$ .\\
We have \[
  \Phi_2(\xi)=|x|(\rho\xi-\lambda)]+t(\lambda-\rho\xi)^a=\Phi_{\lambda,\rho}(\xi,|x|,t)=\Phi(\xi)
  \]
  and we also have\[
  S^2_tf(x)=\left(\int\e^{i\Phi(\xi)}\Lambda(\xi)g(\xi)d\xi\right)\,|x|^{1/2-n/2}
  \]
  where $\Lambda(\xi)=c(\lambda-\rho\xi)^{n/2-1/2}$.
  Choosing $t=t_{k(|x|),j}$ we obtain
  \begin{align}
  |x|^{n/2-1/2}|S^2_{t_{k(|x|),j)}}f(x)|\ge\int g\,\Lambda\, d\xi-\int |\e^{i\Phi}-1|\,g\,\Lambda\, d\xi\\
  \ge\int g\,\Lambda\, d\xi-\max\limits_{|\xi|\le1/2}|\e^{i\Phi(\xi)}-1|\int g\,\Lambda\, d\xi\ge\frac12\int g\,\Lambda\, d\xi\
 \end{align}
 for $|x|\in J_j$ since $\max\limits_{|\xi|\le1/2}|\e^{i\Phi(\xi)}-1|\le\frac12$ according to inequality \eqref{eq:7}.\\
 It follows that\[
 \sup\limits_k|S^2_{t_k}f(x)|\ge\frac12\int g\,\Lambda\, d\xi\,|x|^{1/2-n/2}\ge c \lambda^{n/2-1/2}|x|^{1/2-n/2}
 \]
 for $|x|\in J_j$.
 It remains to study $S_t^1f(x)$. We set $h=g\Lambda$ an then have $h\lesssim\lambda^{n/2-1/2}$ 
 and $|h^\prime|\lesssim(\lambda+\rho)\lambda^{n/2-3/2}\le2\lambda^{n/2-1/2}$. Integrating by parts
 we obtain
 \begin{align}
   |x|^{n/2-1/2}S^1_tf(x)&=\int\e^{i\Phi_1}h\,d\xi=\int\e^{i\Phi_1}i\Phi_1^\prime\frac1{i\Phi_1^\prime}h\,d\xi\\
   &=-\int\e^{i\Phi_1}\frac1i\left(\frac1{\Phi_1^\prime}h^\prime-\frac{\Phi_1''}{(\Phi^\prime)^2} h\right)\,d\xi
 \end{align}
 We have \[
\Phi_1'= -|x|\rho-\rho at(\lambda-\rho\xi)^{a-1}
 \]  
 and
 \[
 \Phi_1''=t a(a-1)\rho^2(\lambda-\rho\xi)^{a-2}
 \]
 and it follows that $|\Phi_1'|\ge\rho[x] $ and  $|\Phi_1'|\ge\rho at(\lambda-\rho\xi)^{a-1} $.
 We have\[
 \frac1{|\Phi_1'|}\le\frac1{\rho|x|}
 \]
 and\[
 \frac{|\Phi_1''|}{|\Phi'|^2}= \frac{1}{|\Phi_1'|} \frac{|\Phi_1''|}{|\Phi_1'|}\le\frac1{\rho|x|}\frac{a|a-1|t\rho^2(\lambda-\rho\xi)^{a-2}}{at(\lambda-\rho\xi)^{a-1} }\lesssim\frac1{\rho|x|}\frac\rho{\lambda-\rho\xi}\lesssim \frac1{\rho|x|}. 
 \] 
 It follows that
 \begin{align}
 |x|^{n/2-1/2}|S^1_tf(x)|\lesssim\frac1{\rho|x|}\int(|h|+|h'|)\,d\xi \lesssim\frac1{\rho|x|}\lambda^{n/2-1/2}.
 \end{align}
 and if $|x|\in J_j$ we get\[
 |S_t^1f(x)|\lesssim|x|^{1/2-n/2}{\frac1{\rho_j\lambda^{a-1}_jb_j }} \lambda^{n/2-1/2}\lesssim\delta\lambda^{n/2-1/2}|x|^{1/2-n/2},
 \]
 where we have used \eqref{eq:8}.
 Hence we have\[
 S^*f(x)=\sup\limits_k|S_{t_k}f(x)|\ge c\lambda^{n/2-1/2}|x|^{1/2-n/2}
 \]
 for $|x|\in J_j$.\\
  The theorem will follow if we show that with $f=f_j$ we have\[
 \frac{\|S^*f_j\|_2}{\|f_j\|_{H_s}}\to\infty \mbox{ as }j\to\infty.
 \]
 We have\[
 \int\limits_{R^n}|S^*f(x)|^2\,dx\ge\int\limits_{|x|\in J_j}|S^*f(x)|^2|\,dx \gtrsim\int\limits_{|x|\in J_j}\lambda^{n-1}|x|^{1-n}\,dx \ge |I_j|\lambda^{n-1} 
 \]
 and\[
 \|f\|^2_{H_s}\lesssim
 \rho^{-1}\lambda^{2s+n-1}.
 \]
 With $f=f_j$ we get \[
 \left(\frac{\|S^*f_j\|_2}{\|f_j\|_{H_s}}\right)^2\gtrsim\frac{\lambda^{n-1}|I_j|}{\rho_j^{-1}\lambda_j^{2s+n-1} }=\rho_j\lambda_j^{-2s}|I_j|. 
\]
We have $|I_j|=a\lambda_j^{a-1}b_j/2$ and obtain
\begin{align}
 \left(\frac{\|S^*f_j\|_2}{\|f_j\|_{H_s}}\right)^2&\gtrsim\rho_j\lambda_j^{-2s} \lambda_j^{a-1}b_j=\rho_j\lambda_j^{a-1-2s}b_j\\
&=\epsilon M_j^{\frac{2-a}a}b_j^{-\frac{1-2s}{a-4s}}\left(M_j^{\frac2a}b_j^{-\frac1{a-4s}}\right)^{a-1-2s}b_j\\
&=\epsilon M_j^{\frac{a-4s}{a}}b_j^{\frac{-1+2s-a+1+2s+a-4s}{a-4}}=\epsilon M_j^{\frac{a-4s}{a}}.
\end{align}
Since $a-4s>0$ and $M_j\to\infty$ as $j\to\infty$ we conclude that\[
\frac{\|S^*f_j\|_2}{\|f_j\|_{H_s}}\to\infty \mbox{ as }j\to\infty,
\]
This completes the proof of theorem .
\end{proof}
\par
Now let $a>0,a\ne1, 0<s<a/4$ and $r=2s/(a-4s)$. Also let $(t_m)_1^\infty$ satisfy \eqref{eq:1} and 
let $(t_m-t_{m+1})_1^\infty$ be decreasing.  \\
It is proved in Dimou and Seeger \cite{Dim-See} that in the case $n=1$ one has\[
\|\sup\limits|S_{t_m}f|\|_2\lesssim\|f\|_{H_s}, \,f\in\mathscr{S}(\R)
\]
if $(t_m)_1^\infty\in l^{r,\infty}$.\\
It then follows from Theorems 5 and 9 that in the case $n\ge2$ one has\[
\|\sup\limits_m|S_{t_m}f|\|_2\lesssim\|f\|_{H_s} 
\]
for all radial functions $f$ in $\mathscr{S}(\R^n)$, if and only if 
  $(t_m)_1^\infty\in l^{r,\infty}$.
%{\color{green} Ej refererat till    \eqref{eq:12}} \\
\par
%%%%%%%%%%%%%%%%%%%
%$============$\\

\Addresses

\end{document}